\documentclass[10pt, a4paper]{amsart}

\usepackage[dvipsnames]{xcolor}
\usepackage{accents}
\usepackage{float, wrapfig}
\usepackage{aliascnt}
\usepackage{amsmath, amsthm,  amsfonts, etoolbox, changepage,lipsum, bigints, relsize, mathtools, color, tikz-cd, amssymb, tikz, multirow ,adjustbox, tabularx, rotating, booktabs, lscape,xcolor,colortbl}
\usepackage{aliascnt}
\usepackage{url}

\usepackage[english]{babel}
\usepackage[utf8]{inputenc}
\usepackage{csquotes}

\input xypic
\xyoption{all}
\usepackage[all,arc]{xy}
\usepackage{enumerate}
\usepackage{mathrsfs}
\usepackage{mathtools} 
\usetikzlibrary{decorations.markings}

\usepackage{hyperref}
\hypersetup{
    colorlinks=true,
    linkcolor=blue,
    filecolor=magenta,      
    urlcolor=cyan,
    citecolor = blue
} 
\usepackage{cleveref}

\usepackage[style=alphabetic]{biblatex}
\usepackage[OT2,OT1]{fontenc}
\newcommand\cyr
{
\renewcommand\rmdefault{wncyr}
\renewcommand\sfdefault{wncyss}
\renewcommand\encodingdefault{OT2}
\normalfont
\selectfont
}
\DeclareTextFontCommand{\textcyr}{\cyr}

\usepackage{comment}
\usetikzlibrary{calc, decorations.pathreplacing,angles,quotes}

\usepackage{caption}
\usepackage{subcaption}

\pgfkeys{tikz/mymatrixenv/.style={decoration={brace},every left delimiter/.style={xshift=8pt},every right delimiter/.style={xshift=-8pt}}}
\pgfkeys{tikz/mymatrix/.style={matrix of math nodes,nodes in empty cells,left delimiter={[},right delimiter={]},inner sep=1pt,outer sep=1.5pt,column sep=2pt,row sep=2pt,nodes={minimum width=20pt,minimum height=10pt,anchor=center,inner sep=0pt,outer sep=0pt}}}
\pgfkeys{tikz/mymatrixbrace/.style={decorate,thick}}

  \theoremstyle{plain}
  \newtheorem{theorem}{Theorem}[section]
 \newtheorem{lemma}[theorem]{Lemma}

 \newtheorem{example}[theorem]{Example}  
 \newtheorem{corollary}[theorem]{Corollary}

\newtheorem{proposition}[theorem]{Proposition}
\crefname{lemma}{Lemma}{Lemma}
  \crefname{corollary}{Corollary}{Corollary}
  \crefname{theorem}{Theorem}{Theorem}
  \crefname{definition}{Definition}{Definition}
   \crefname{proposition}{Proposition}{Proposition}
 \crefname{section}{Section}{Section} 
   \crefname{construction}{Construction}{Construction}
   \crefname{generalization}{Generalization}{Generalization}
  \crefname{construction}{Construction}{Construction}
  \crefname{notation}{Notation}{Notation}
   \crefname{example}{Example}{Example}
  \crefname{remark}{Remark}{Remark}
  \crefname{fact}{Fact}{Fact}
  \crefname{conjecture}{Conjecture}{Conjecture}
  \crefname{motivation}{Motivation}{Motivation}  
  \crefname{figure}{Figure}{Figure}  
  
  \newtheorem{definition}[theorem]{Definition}
  
  \newtheorem{remark}[theorem]{Remark}
  \newtheorem{note}{Note}[section]

  \numberwithin{equation}{section}
 \numberwithin{figure}{section}
  \numberwithin{figure}{subsection}

  \newcommand{\cA}{{\mathcal A}}

  \renewcommand{\cD}{{\mathcal D}}

  \newcommand{\cK}{{\mathcal K }}

  \newcommand{\cB}{{\mathcal B }}

       \newcommand{\cU}{{\mathcal U }}

  \newcommand{\Hom}{\text{Hom}}

  \newcommand{\Kom}{\text{Kom}}
  
  \newcommand{\id}{\text{id}}
  \newcommand{\Ba}{\mathscr{B}}
  \newcommand{\Aa}{\mathscr{A}}

   \newcommand{\ba}{\begin{eqnarray}}
   \newcommand{\na}{\end{eqnarray}}
   \newcommand{\ban}{\begin{eqnarray*}}
   \newcommand{\nan}{\end{eqnarray*}}

  \newcommand{\B}{\mathbb B}
  \newcommand{\C}{{\mathbb C}}
  \newcommand{\R}{{\mathbb R}}
  \newcommand{\Z}{{\mathbb Z}}

       \newcommand{\bS}{\mathbb S}

  \newcommand{\sB}{\mathscr{B}}

\newcommand{\ra}{\rightarrow}

\newcommand{\xra}{\xrightarrow}

\newcommand\rb[1]{\textcolor{red}{\textbf{#1}}}

\newcommand{\ol}{\overline}

     \def\uv{\underline{v}}
       \def\uw{\underline{w}}

  \newcommand{\<}{\langle}
  \renewcommand{\>}{\rangle}

\usepackage{scalerel,stackengine}
\stackMath
\newcommand\reallywidehat[1]{%
\savestack{\tmpbox}{\stretchto{%
  \scaleto{%
    \scalerel*[\widthof{\ensuremath{#1}}]{\kern-.6pt\bigwedge\kern-.6pt}%
    {\rule[-\textheight/2]{1ex}{\textheight}}
  }{\textheight}%
}{0.5ex}}%
\stackon[1pt]{#1}{\tmpbox}%
}
\newtheorem*{theorem*}{Theorem}

\newtheoremstyle{named}{}{}{\itshape}{}{\bfseries}{.}{.5em}{\thmnote{#3's }#1}
\theoremstyle{named}

\newtheoremstyle{name}{}{}{\itshape}{}{\bfseries}{.}{.5em}{\thmnote{#3}#1}
\theoremstyle{name}

\makeatletter
\newcommand\xrightleftarrows[2][]{\ext@arrow 0099{\longrightleftarrowsfill@}{#1}{#2}}
\def\longrightleftarrowsfill@{\arrowfill@\leftarrow\relbar\rightarrow}
\makeatother

\title{
Faithfulness of the 2-braid group via zigzag algebra in type B}
\author{Edmund Heng and Kie Seng Nge}

\address{Institut des Hautes Études Scientifiques, 35, Route de Chartres, 91440 Bures-sur-Yvette, France}
\email{heng@ihes.fr}

\address{School of Mathematics and Physics, Department of Mathematics and Applied Mathematics, Xiamen University Malaysia, Block A4, Jalan Sunsuria, Bandar Sunsuria, 43900 Sepang, Selangor Darul Ehsan, Malaysia.}
\email{kieseng.nge@xmu.edu.my}



\calclayout

\begin{document}

\begin{abstract}
We show that a certain category of bimodules over a finite dimensional quiver algebra known as type $B$ zigzag algebra is a quotient category of the category of type $B$ Soergel bimodules.
This leads to an alternate proof of Rouquier's conjecture on the faithfulness of the 2-braid groups for type $B$.
\end{abstract}

\maketitle

\section{Introduction}
In \cite{rouquier_2006}, Rouquier introduced the \emph{2-braid groups} associated to the Artin braid groups using Rouquier complexes, which are certain complexes built from Soergel bimodules.
As the category of Soergel bimodules is a categorification of the Hecke algebra \cite{Soergel92, Soergel_07}, the 2-braid group can be viewed as a categorification of the Hecke algebra representation of the corresponding Artin braid group.
Within the same paper \cite{rouquier_2006}, Rouquier stated the faithfulness of the 2-braid group as a conjecture, where in type $A$ faithfulness is known due to the results in \cite{KhoSei}.
This was extended to the ADE types in \cite{brav_thomas_2010} and later on to all other finite types in \cite{jensen_2016}.

In this paper we shall provide an alternate proof to faithfulness of the 2-braid groups in the type $B(=C)$ case.
In contrast to the proofs in \cite{brav_thomas_2010} and \cite{jensen_2016}, our proof is closer in spirit to the proof of the type $A$ case in \cite{rouquier_2006}, which we now briefly recall.
In \cite{KhoSei}, the authors provide a nil-categorification of the type $A$ Temperley-Lieb algebra, using certain monoidal, additive subcategory of $\Aa$-bimodules for some quiver algebra $\Aa$ known as the type $A$ zigzag algebra.
Moreover, certain complexes over these $\Aa$-bimodules (almost identical to Rouquier complexes) collectively define a braid group action on the homotopy category of projective modules over $\Aa$.
This action categorifies the Burau representation of the type $A$ braid group and is shown to be faithful.
Just as the Temperley-Lieb algebra is a quotient of the Hecke algebra, this subcategory of $\Aa$-bimodules can be viewed as a quotient category of the category of Soergel bimodules; namely, one can construct an essentially surjective functor from the category of Soergel bimodules to this subcategory of $\Aa$-bimodules.
As such, the faithfulness of the type $A$ 2-braid group\footnote{This is not to be confused with 2-faithfulness of a braid group involving cobordisms, see \cite[Section 3.3]{KT07}.} follows from the faithfulness of the categorified Burau representation.

In our previous work \cite{heng_nge}, a type $B$ analogue of \cite{KhoSei} was developed.
In particular, we constructed a quiver algebra  $\Ba$ that we called type $B$ zigzag algebra\footnote{There is a different zigzag algebra that can be thought of as type $C$; see \cref{rmk:typeCzigzag}.}, such that the type $B$ Artin braid group acts faithfully on the homotopy category of projective modules over $\Ba$ using certain complexes of $\Ba$-bimodules (which again, is almost identical to the Rouquier complexes).
As such, it is only natural to show that this whole story for type $B$ also fits into the world of 2-braid groups and Soergel-bimodules; such is the goal of this paper.

To be more precise, we shall describe a monoidal, additive subcategory of $\Ba_n$-bimodules and construct an essentially surjective functor from the category of type $B_n$ Soergel bimodules.
As in the case for type $A$, this will produce an alternate proof of the faithfulness of the type $B$ 2-braid groups.
This result also partly serves as evidence that the zigzag algebra we introduced in \cite{heng_nge} is indeed a reasonable definition of a type $B$ zigzag algebra, at least from the point of view of braid group actions.

\subsection*{Outline of the paper}
We shall follow closely the work of Jensen in \cite{jensen_master}, which spells out the details for type $A$. 
In particular, all of our results will be formulated using the equivalent diagrammatic category of Soergel bimodules given in \cite{EW_2016}.

\cref{define zigzag} and \cref{Soergel} are brief summaries of type $B$ zigzag algebras $\Ba$ and 2-braid groups respectively, where in \cref{define zigzag} we also define the relevant category of $(\Ba, \Ba)$-bimodules $\cB$ (\cref{defn: category of type B bimodules}) that categorifies a quotient of the type $B$ Temperley-Lieb algebra (in the sense of \cite{Green}, see \cref{Functorial TL}).
The main results of this paper and their proofs are all contained in \cref{quotient and proof}, where we construct the quotient functor from the category of Soergel bimodules to the category of $(\Ba,\Ba)$-bimodules $\cB$.
The faithfulness of the 2-braid group will then be a simple consequence.

\subsection*{Acknowledgements}
We would like to thank our supervisor Tony Licata for his guidance throughout.
This paper results from a discussion with Thorge Jensen during the conference ``New Connections in Representation Theory 2020, Mooloolaba'', of which we are deeply grateful for.
We thank the organisers for making this possible and we thank Thorge for his patience in answering our questions.
We would also like to thank the referee for their suggestion and for carefully reading our paper.

\section{Type B Zigzag Algebras}\label{define zigzag}

In this section we recall the relevant objects and results from \cite{heng_nge}: the construction of the type $B_n$ zigzag algebra $\Ba_n$ and the fact that the type $B_n$ Artin group $\mathcal{A}(B_n)$ acts faithfully on the bounded homotopy category $\Kom^b(\Ba_n$-$\text{p$_{r}$g$_{r}$mod})$ of projective, $\Z$-graded modules over $\Ba_n$.
We shall also describe the additive and monoidal category of $\Ba_n$-bimodules that categorifies a quotient of the type $B$ Temperley-Lieb algebra (in the sense of \cite{Green}).

\subsection{Type \texorpdfstring{$B_n$}{B} Artin group}
For $n \geq 2,$ the type $B_n$ Artin group $\cA({B_n})$ is the group generated by $n$ generators:
$$ \sigma_1^B,\sigma_2^B,\ldots, \sigma_{n}^B $$
\noindent subject to the relations
\begin{alignat}{2}
\sigma_1^B \sigma_2^B \sigma_1^B \sigma_2^B &= \sigma_2^B \sigma_1^B \sigma_2^B \sigma_1^B;   \label{length 4 relation} \\
      \sigma_j^B  \sigma_k^B  &= \sigma_k^B  \sigma_j^B,    &&\text{for }  |j-k|> 1; \label{commuting relation} \\    
      \sigma_j^B  \sigma_{j+1}^B  \sigma_j^B  &=  \sigma_{j+1}^B  \sigma_{j}^B  \sigma_{j+1}^B,  \quad &&\text{for }  j= 2,3, \ldots, n-1. \label{length 3 relation}
\end{alignat}
Its corresponding Coxeter group $W(B_n)$ is generated by $s_1,s_2,...,s_n$ subject to the relations \eqref{length 4 relation}, \eqref{commuting relation}, \eqref{length 3 relation} (with $s$ in place of $\sigma^B$) and $s_j^2 = 1$ for all $j \geq 1$.

\subsection{Type \texorpdfstring{$B_n$}{B} zigzag algebra \texorpdfstring{$\Ba_n$}{}} \label{B zigzag}
Consider the following quiver $Q_n$:
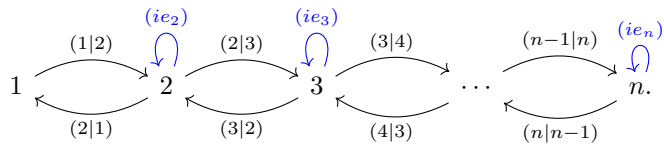
\begin{figure}[H]
\begin{tikzcd}[column sep = 1.5cm]
1				\arrow[r,bend left,"(1|2)"]  														 &		 
2 			\arrow[l,bend left,"(2|1)"] \arrow[r,bend left, "(2|3)"] 
				\arrow[color=blue,out=70,in=110,loop,swap,"(ie_2)"] &
3 			\arrow[l,bend left,"(3|2)"] \arrow[r,bend left, "(3|4)"] 
				\arrow[color=blue,out=70,in=110,loop,swap,"(ie_3)"] &
\cdots \arrow[l,bend left,"(4|3)"] \arrow[r,bend left, "(n-1|n)"] &
n 		.	\arrow[l,bend left,"(n|n-1)"]
				\arrow[color=blue,out=70,in=110,loop,swap,"(ie_n)"]
\end{tikzcd}
\caption{{\small The quiver $Q_n$. Note that paths are read from left to right. The loops $(ie_j)$ are exceptionally length 0.}}
\label{B quiver}
\end{figure}
Take its path algebra $\R Q_n$ over $\R$ and consider the path length grading on $\R Q_n$, where exceptionally the ``imaginary'' path $(ie_j)$ has grading $0.$ 
Note that $(ie_j)$ is the loop in \cref{B quiver} and is not to be confused with the constant path $e_j$ (also length 0).
In this paper we use the notation $(1)$ to denote a grading shift \emph{down} by $1.$

We are now ready to define the zigzag algebra of type $B_n$:
\begin{definition}
The zigzag path algebra of $B_n$, denoted by $\Ba_n$, is the quotient algebra of the path algebra $\R Q_n$ modulo the usual zigzag relations given by
\begin{align}
(j|j-1)(j-1|j) &= (j|j+1)(j+1|j) \qquad (=: X_j);\\
(j-1|j)(j|j+1) = & 0 = (j+1|j)(j|j-1);
\end{align}
for $2\leq j \leq n-1$, in addition to the relations
\begin{align}
(ie_j)(ie_j) &= -e_j, \qquad \text{for } j \geq 2 \label{imaginary};\\
(ie_{j-1})(j-1|j) &= (j-1|j)(ie_j), \qquad \text{for } j\geq 3; \label{complex symmetry 1}\\
(ie_{j})(j|j-1) &= (j|j-1)(ie_{j-1}), \qquad \text{for } j\geq 3; \label{complex symmetry 2}\\
(1|2)(ie_2)(2|1) &= 0, \\
(ie_2) X_2 &= X_2 (ie_2).
\end{align}
\end{definition}
\begin{remark} \label{rmk:typeCzigzag}
    One can define a ``type C'' zigzag algebra by switching the position of the (length zero) loops $(ie_j)$, where we only have $(ie_1)$ on vertex 1 and nowhere else. The relations are defined similarly. 
\end{remark}
Since the relations are all homogeneous with respect to the path length grading, $\Ba_n$ is a $\Z$-graded algebra.
As a $\R$-vector space, $\Ba_n$ has dimension $8n-6$, with the following basis:
\begin{small}
\begin{align*}
\{ 
    &e_1 , \ldots, e_{n}, ie_2 , \ldots, ie_{n},  \\
    & (1|2), \ldots, (n-1 | n), (2|1), \ldots, (n | n-1), (ie_2)(2|1), (1|2)(ie_2), (ie_2)(2|3), \ldots, (ie_{n-1})(n-1|n), \\
    & (3|2)(ie_2), \ldots, (n|n-1)(ie_{n-1}), (1|2|1), \ldots, (2n-1|2n-2|2n-1), (ie_2)(2|1|2), \ldots, (ie_n)(n|n-1|n) 
    \}.
\end{align*}
\end{small}

The indecomposable (left) projective $\Ba_n$-modules are given by $P^B_j := \Ba_n e_j$.
For $j=1$, $P^B_j$ is naturally a $(\Ba_n, \R)$-bimodule; there is a natural left $\Ba_n$-action given by multiplication of the algebra and the right $\R$-action induced by the natural $\R$-vector space structure.
Nonetheless, for $j\geq 2$, we shall endow $P^B_j$ with a right $\C$-action.
To this end, note that \eqref{imaginary} is analogous to the relation satisfied by the complex imaginary number $i$.
We define a right $\C$-action on $P^B_j$ by $p * (a+ib) = ap + bp(ie_j)$ for $p \in P^B_j, a+ib\in \C$.
Further note that this right action restricted to $\R$ agrees with both the natural right and left $\R$-action.
This makes $P^B_j$ into a $(\Ba_n,\C)$-bimodule for $j\geq 2$.
Dually, we shall define ${}_jP^B := e_j\Ba_n$, where we similarly consider it as a ($\R,\Ba_n)$-bimodule for $j=1$ and as a $(\C,\Ba_n)$-bimodule for $j\geq 2$.

It is easy to check that we have the following isomorphisms of $\Z$-graded bimodules:
\begin{proposition}[{\cite[Proposition 3.5]{heng_nge}}]\label{iPj identification}
Denote ${}_jP^B_k := {}_jP^B\otimes_{\Ba_n}P^B_k$. We have that
\[
  {}_jP^B\otimes_{\Ba_n}P^B_k \cong 
  \begin{cases}
  		\ _{\C}\C_{\C}(-1), & \text{as graded } (\C,\C)\text{-bimodules, for } j,k \in \{2,\hdots, n\}, |j-k|=1;\\
  		\ _{\C}\C_{\C} \oplus {}_{\C}\C_{\C}(-2), & \text{as graded } (\C,\C)\text{-bimodules, for } j=k=2,3,\hdots,n; \\
        \ _{\R}\C_{\C}(-1), & \text{as graded } (\R,\C)\text{-bimodules, for } j=1 \text{ and } k=2; \\
        \ _{\C}\C_{\R}(-1), & \text{as graded } (\C,\R)\text{-bimodules, for } j=2 \text{ and } k=1; \\ 
        \ _{\R}\R_{\R} \oplus {}_{\R}\R_{\R}(-2), & \text{as graded } (\R,\R)\text{-bimodules, for } j=k=1; \\    
        \ 0, &\text{otherwise}.
  \end{cases}
\]
\label{bimodule isomorphism}
\end{proposition}
\begin{remark}
Note that all the graded bimodules in \cref{bimodule isomorphism} can be restricted to a $(\R, \R)$-bimodule by identifying ${}_\R \C_\R \cong \R\oplus \R$ as $(\R, \R)$-bimodules.
For example, ${}_1 P_2^B$ as a $(\R, \R)$-bimodule is generated by $(1|2)$ and $(1|2)i$, so it is isomorphic to $\R(-1) \oplus \R(-1) \cong {}_\R \C_\R (-1)$.
\end{remark}

\begin{lemma}[{\cite[Lemma 3.7]{heng_nge}}]
Denote $\mathbb{K}_1 := \R$ and $\mathbb{K}_j := \C$ when $j \geq 2$.
The maps \[\beta_j: P^B_j \otimes_{\mathbb{K}_j} {}_jP^B  \to \Ba_n \text{ and } \gamma_j: \Ba_n  \to  P^B_j \otimes_{\mathbb{K}_j} {}_jP^B  (2) \] defined by:
\begin{align*} 
\beta_j(x\otimes y) &:= xy, \\
\gamma_j(1) &:= 
\begin{cases}
X_j \otimes e_j + e_j \otimes X_j + (j+1|j) \otimes (j|j+1) \\
 \hspace{8mm} + (-ie_{j+1})(j+1|j) \otimes (j|j+1)(ie_{j+1}), &\text{for } j=1;\\
X_j \otimes e_j + e_j \otimes X_j + (j-1|j) \otimes (j|j-1) + (j+1|j) \otimes (j|j+1), &\text{for } 1<j < n; \\
X_j \otimes e_j + e_j \otimes X_j + (j-1|j) \otimes (j|j-1), &\text{for } j = n,
\end{cases}
\end{align*}
are $(\Ba_n,\Ba_n)$-bimodule maps.
\end{lemma}

\begin{definition} \label{beta and gamma maps}
Define the following complexes of graded $(\Ba_n,\Ba_n)$-bimodules:
\begin{align*}
R_j &:= (0 \to P^B_j \otimes_{\mathbb{K}_j} {}_jP^B \xra{\beta_j} \Ba_n \to 0), \text{and} \\
R_j' &:= (0 \to \Ba_n  \xra{\gamma_j}  P^B_j \otimes_{\mathbb{K}_j} {}_jP^B(2) \to 0).
\end{align*}
for each $j \in \{1,2, \cdots, n\},$ with both $\Ba_n$ in cohomological degree 0, $\mathbb{K}_1 = \R$ and $\mathbb{K}_j = \C$ for $j \geq 2$.
\end{definition}

\begin{theorem}[{\cite[Theorem 3.13]{heng_nge}}] \label{Cat B action}
Let $\Kom^b(\Ba_n$-$p_r g_r mod)$ denote the homotopy category of complexes of projective, graded left $\Ba_n$-modules.
We have a (weak) $\mathcal{A}(B_n)$-action on $\Kom^b(\Ba_n$-$p_r g_r mod)$, where each standard generator $\sigma^B_j$ for $j \geq 1$ of $\mathcal{A}(B_n)$ acts on a complex $M \in \Kom^b(\Ba_n$-$p_rg_rmod)$ via $R_j$:
$$\sigma^B_j(M):= R_j \otimes_{\Ba_n} M,\text{ and } \quad  (\sigma^B_j)^{-1}(M):= R_j' \otimes_{\Ba_n} M,$$
\end{theorem}

\subsection{Nil-categorification of the type \texorpdfstring{$B$}{B} Temperley-Lieb algebra} \label{Functorial TL}
Recall that the type $B_n$ Temperley-Lieb algebra $TL_v(B_n)$ over $\Z[v,v^{-1}]$ (in the sense of \cite[Proposition 1.3]{Green}) can be described explictly as the algebra generated by $E_1, ..., E_n$ with the following relations:
\begin{align*}
E_j^2 &= vE_j + v^{-1}E_j; \\
E_j E_k &= E_k E_j, \quad \text{if } |j-k| > 1; \\
E_j E_k E_j &= E_k, \quad \text{if } |j-k| = 1 \text{ and } j,k > 1; \\
E_j E_k E_j E_k &= 2E_j E_k, \quad \text{if } \{j, k\} = \{1,2\}.
\end{align*}

The following bimodules that satisfy a further quotient of the relations above (see \eqref{2}):
\begin{proposition} \label{prop: potential categorification TL(B)}
Define $\cU_j := P_j^B \otimes_{\mathbb{K}_j} {}_j P^B (1)$, where $\mathbb{K}_1 = \R$ and $\mathbb{K}_j = \C$ when $j \geq 2$.
The following are isomorphic as $\Z$-graded $(\Ba_n, \Ba_n)$-bimodules:
\begin{align}
\label{1} \cU_j \otimes_{\Ba_n} \cU_j &\cong \cU_j(1) \oplus \cU_j(-1); \\
\label{2} \cU_j \otimes_{\Ba_n} \cU_k &\cong 0, \quad \text{if } |j-k| > 1; \\
\label{3} \cU_j \otimes_{\Ba_n} \cU_k \otimes_{\Ba_n} \cU_j &\cong \cU_k, \quad \text{if } |j-k| = 1 \text{ and } i,j > 1; \\
\label{4} \cU_j \otimes_{\Ba_n} \cU_k \otimes_{\Ba_n} \cU_j \otimes_{\Ba_n} \cU_k &\cong (\cU_j \otimes_{\Ba_n} \cU_k) \oplus (\cU_j \otimes_{\Ba_n} \cU_k), \quad \text{if } \{j, k\} = \{1,2\}.
\end{align}
\end{proposition}
\begin{proof}
The isomorphisms \eqref{1}, \eqref{2} and \eqref{3} follow from the exact same proof in the type $A$ case \cite[Theorem 2.2]{KhoSei}.
  For the fourth isomorphism \eqref{4} with $j= 1, k=2$ (the other case is similar), we have the following chain of isomorphisms:
\begin{align*}
 &\cU_1 \otimes_{\Ba_n} \cU_2 \otimes_{\Ba_n} \cU_1 \otimes_{\Ba_n} \cU_2 \\
 \cong &  P_1^B \otimes_{\mathbb{R}} ( {}_1 P^B   \otimes_{\Ba_n} P_2^B) \otimes_{\mathbb{C}} ({}_2 P^B  \otimes_{\Ba_n} P_1^B) \otimes_{\mathbb{R}} {}_1 P^B  \otimes_{\Ba_n}P_2^B \otimes_{\mathbb{C}} {}_2 P^B (4)  \\
 \cong &P_1^B \otimes_{\mathbb{R}} ({}_\mathbb{R} \mathbb{C}_\mathbb{C} \otimes_{\mathbb{C}} {}_\mathbb{C} \mathbb{C}_\mathbb{R}) \otimes_{\mathbb{R}} {}_1 P^B  \otimes_{\Ba_n}P_2^B \otimes_{\mathbb{C}} {}_2 P^B (2) \\
 \cong &P_1^B \otimes_{\mathbb{R}} (\mathbb{R} \oplus \mathbb{R}) \otimes_{\mathbb{R}} {}_1 P^B  \otimes_{\Ba_n}P_2^B \otimes_{\mathbb{C}} {}_2 P^B (2) \\
 \cong &\left(P_1^B \otimes_{\mathbb{R}} {}_1 P^B  \otimes_{\Ba_n}P_2^B \otimes_{\mathbb{C}} {}_2 P^B (2) \right)  \oplus \left(P_1^B \otimes_{\mathbb{R}} {}_1 P^B  \otimes_{\Ba_n}P_2^B \otimes_{\mathbb{C}} {}_2 P^B (2) \right) \\
 \cong &(\cU_1 \otimes_{\Ba_n} \cU_2) \oplus (\cU_1 \otimes_{\Ba_n} \cU_2),
\end{align*}  
where we have used \cref{iPj identification} repeatedly. 
\end{proof}

\begin{definition}\label{defn: category of type B bimodules}
Let $\cB$ denote the monoidal category over $- \otimes_{\Ba_n} -$ generated by the $\cU_j$ defined in \cref{prop: potential categorification TL(B)}, with monoidal unit $\Ba_n$.
We use $\cK ar(\ol{\cB})$ to denote the Karoubi envelope of the additive closure $\ol{\cB}$ of $\cB$.
\end{definition}
\begin{corollary}
The monoidal, additive category $\cK ar(\ol{\cB})$ categorifies a quotient of $TL_v(B_n)$
\end{corollary}
\begin{proof}
This follows directly from comparing \cref{prop: potential categorification TL(B)} with the relations of the type $B$ Temperley-Lieb algebra $TL_v(B_n)$.
\end{proof}

\section{Soergel bimodules and the 2-braid group} \label{Soergel}
In this section, we describe the diagrammatic version of the category of Soergel bimodules as given in \cite{EW_2016}.
We will only state the minimal details required to understand the category in the type $B$ case and only with respect to a particular realisation (see next paragraph).
For the general theory, we refer the reader to \cite{EW_2016} and the references therein.
We shall also recall the definition of the 2-braid group from \cite{rouquier_2006} in the type $B$ case.

Throughout this section, $(W, S=\{s_1,s_2,...,s_n\})$ shall denote the type $B_n$ Coxeter system, with $W$ the type $B_n$ Coxeter group.
We shall fix the following balanced, but \emph{non-symmetric} realisation of $(W, S)$: we define $\mathfrak{h} := \bigoplus_{s_i \in S} \R\alpha_{s_i}^\vee$ with the set of coroots $\{\alpha_{s_i}^\vee : s_i \in S\} \subset \mathfrak{h}$ and roots $\{\alpha_{s_i} : s_i \in S\} \subset \mathfrak{h}^* = \Hom_\R(\mathfrak{h}, \R)$, where
\[
a_{s_i,s_j}:= \<\alpha_{s_i}^\vee, \alpha_{s_j}\> = 
\begin{cases}
2, &\text{ if } i=j; \\
-1, &\text{ if } i,j \geq 2, |i-j| = 1; \\
-1, &\text{ if } i=1, j=2; \\
-2, &\text{ if } i=2, j=1; \\
0, &\text{ otherwise}.
\end{cases}
\]
To this realisation, we shall set $R:= \bigoplus_{k \geq 0} S^k(\mathfrak{h}^*)$ as the $\Z$-graded symmetric $\R$-algebra on $\mathfrak{h}^*$ with deg$(\mathfrak{h}^*)$=2.
Note that $W$ acts naturally on $\mathfrak{h}^*$ and hence acts on $R$.

\subsection{Soergel Bimodules as a Diagrammatic Category}
Let $\cD_S,$ or simply $\cD$, denote the $\R$-linear monoidal category (with respect to the realisation before) defined as follows:
\begin{enumerate} [i.]
\item Objects: finite sequences $\underline{w} = s_{i_1}s_{i_2}\hdots s_{i_k}$  of letters in $S$, with monoidal structure given by concatenation and monoidal unit given by the empty sequence $\emptyset$.
\item  Homomorphisms: the Hom space $Hom_{\cD}(\uw, \uv)$ is the $\R$-span of Soergel graphs decorated by homogeneous $f \in R$ modulo the relations listed below, where the graphs have bottom boundary $\uw$ and top boundary $\uv$ (see \cite[Definition 5.1]{EW_2016} for the precise definition).
In our type $B$ cases, the possible vertices of the Soergel graphs consist of the following types:
	 \begin{enumerate} [(i)]
	 \item Univalent vertices (dots);
	 \item Trivalent vertices connecting three edges of the \emph{same} colour;
	 \item $2 m_{st}$-valent vertices connecting edges which alternate in colour between two elements $s,t$ of $S$, where in our case $m_{st} = 2, 3$ or $4$.
	 \end{enumerate}
\begin{figure} [H]
\begin{tikzpicture}[scale=0.7]
\draw[ purple, line width=0.8mm] (0,.7) -- (0,0);
\filldraw[purple] (0,.7) circle (3pt) ;
\node at (0,-.7) {\text{(i)}};
\draw[ purple, line width=0.8mm] (2,1.15) -- (2,.5);
\draw[ purple, line width=0.8mm] (2,.5) -- (1.5,0);
\draw[ purple, line width=0.8mm] (2,.5) -- (2.5,0);
\node at (2,-.7) {\text{(ii)}};

\draw[color=purple, line width=0.8mm] (5.3,1.15) -- (4,-.15);
\draw[color=green, line width=0.8mm] (4,1.15) -- (5.3,-.15);
\node at (5.3,-.7) {\text{(iii)} {\small For $m_{\textcolor{purple}{s}\textcolor{green}{t}} = 2$}};
\draw[color=blue, line width=0.8mm] (8,1.15) -- (8.65,.5);
\draw[color=blue, line width=0.8mm] (9.3,1.15) -- (8.65,.5);
\draw[color=blue, line width=0.8mm] (8.65,.5) -- (8.65,-.15);
\draw[color=red, line width=0.8mm] (8.65,.5) -- (9.3,-.15);
\draw[color=red, line width=0.8mm] (8.65,.5) -- (8,-.15);
\draw[color=red, line width=0.8mm] (8.65,.5) -- (8.65,1.15);
\node at (9.3,-.7) {\text{(iv)} {\small For $m_{\textcolor{red}{s}\textcolor{blue}{t}} = 3$}};
\draw[color=blue, line width=0.8mm] (12,1.15) -- (13.3,-.15);
\draw[color=blue, line width=0.8mm] (13.3,1.15) -- (12,-.15);
\draw[color=red, line width=0.8mm] (11.75,0.5) -- (13.5,0.5);
\draw[color=red, line width=0.8mm] (12.65,1.3) -- (12.65,-.3);
\node at (13.3,-.7) {\text{(v)} {\small For $m_{\textcolor{red}{s}\textcolor{blue}{t}} = 4$}};
\end{tikzpicture}
\begin{caption} {The possible vertices in a Soergel graph in type $B$.} \label{fig:Soergelvertex}
\end{caption}
\end{figure}
Each Soergel graph has a \textit{degree}, which accumulates $+1$ for each dots $(i)$, $-1$ for each trivalent vertices $(ii),$ $0$ for each $2m_{st}$-valent vertices $(iii)$, and the degree of each decoration $f\in R$.
The Hom spaces are then graded by the degree of the Soergel graphs. 
 As one easily checks, the relations below are indeed homogeneous. 
We remind the reader that Soergel graphs are by definition invariant under isotopy that preserves the top and bottom boundary, so the bent and rotated versions of the relations below also hold.

\begin{enumerate} [(a)]
\item The polynomial relations: 

\begin{itemize}  
\item {the barbell relation} 
\begin{equation}
 \label{barbell} 
 \centering
\begin{tikzpicture}[scale=0.7]
\draw[dashed,color=black!60] (0,0) circle (1.0);
\draw[dashed,color=black!60] (3,0) circle (1.0);
\draw[ purple, line width=0.8mm] (0,.55) -- (0,-.55);
\filldraw[purple] (0,.55) circle (3pt) ;
\filldraw[purple] (0,-.55) circle (3pt) ;
\node at (4.3, 0) {;};
\node at (1.5,0) {=};
\node at (3,0) {{\LARGE $\alpha_{\textcolor{purple}{s}}$}};
\end{tikzpicture}
\end{equation}
\item {the polynomial forcing relation}
\begin{equation}   \label{polyforce}
\centering
\begin{tikzpicture} [scale = 0.8]
\draw[dashed,color=black!60] (0,0) circle (1.0);
\draw[dashed,color=black!60] (3,0) circle (1.0);
\draw[dashed,color=black!60] (6,0) circle (1.0);
\draw[color=purple, line width=0.8mm] (0,-1) -- (0,1);
\draw[purple, line width=0.8mm] (3,-1) -- (3,1);
\draw[purple, line width=0.8mm] (6,1) -- (6,.55);
\draw[purple, line width=0.8mm] (6,-1) -- (6,-.55);
\filldraw[purple] (6,.55) circle (3pt) ;
\filldraw[purple] (6,-.55) circle (3pt) ;
\node at (1.5,0) {=};
\node at (4.5,0) {+};
\node at (7.5,0) {,};
\node at (-.5,0) {{\LARGE $f$}};
\node at (3.5,0) {{\large $\textcolor{purple}{s} (f)$}};
\node at (6,0) {{\Large $\partial_{\textcolor{purple}{s}} f$}};
\end{tikzpicture}
\end{equation}
where $\partial_{\textcolor{purple}{s}}: R \ra R^{\color{purple}{s}}(-2)$ is the Demazure operator defined by $\partial_{\textcolor{purple}{s}}(f) := \frac{f - {\textcolor{purple}{s}}(f)}{\alpha_{\textcolor{purple}{s}}}$.
\end{itemize}

\item The one colour relations:

\begin{itemize}
\item the needle relation
\begin{equation} \label{needle}
\centering
\begin{tikzpicture} [scale=0.7]
\draw[dashed,color=black!60] (0,0) circle (1.0);
\draw[purple, line width=0.8mm] (0,0) circle (.35);
\draw[purple, line width=0.8mm] (0,-.35) -- (0,-1);
\node at (1.5,0) {=};
  \node at (2,0) {{\Large $0$} ;};
\end{tikzpicture}
\end{equation}
\item the Frobenius relations
\begin{equation} \label{Frobenius}
\centering
\begin{tikzpicture} [scale = .7]
\draw[ purple, line width=0.8mm] (0,.25) -- (0,-.25);
\draw[ purple, line width=0.8mm] (0,.25) -- (-.5,.75);
\draw[ purple, line width=0.8mm] (0,.25) -- (.5,.75);
\draw[ purple, line width=0.8mm] (0,-.25) -- (-.5,-.75);
\draw[ purple, line width=0.8mm] (0,-.25) -- (.5,-.75);

\draw[ purple, line width=0.8mm] (2.75,0) -- (3.25,0);
\draw[ purple, line width=0.8mm] (2.75,0) -- (2.35,0.3);
\draw[ purple, line width=0.8mm] (2.35,0.3) -- (2.35,0.75);
\draw[ purple, line width=0.8mm] (3.25,0) -- (3.65,0.3);
\draw[ purple, line width=0.8mm] (3.65,0.3) -- (3.65,0.75);
\draw[ purple, line width=0.8mm] (2.75,0) -- (2.35,-0.3);
\draw[ purple, line width=0.8mm] (2.35,-0.3) -- (2.35,-0.75);
\draw[ purple, line width=0.8mm] (3.25,0) -- (3.65,-0.3);
\draw[ purple, line width=0.8mm] (3.65,-0.3) -- (3.65,-0.75);
\node at (1.5,0) {=};
\node at (8,0) {\text{general associativity,}};

\end{tikzpicture}
\end{equation}
\begin{equation} \label{wall}
\centering
\begin{tikzpicture} [scale=.7]
\draw[ purple, line width=0.8mm] (0,.75) -- (0,-.75);
\draw[ purple, line width=0.8mm] (0,0) -- (0.45,0);
\filldraw[purple] (0.45,0) circle (3pt) ;
\node at (1,0) {=};
\draw[ purple, line width=0.8mm] (1.75,.75) -- (1.75,-.75);
\node at (6,0) {\text{general unit}.};
\end{tikzpicture}
\end{equation}
\end{itemize}

\item The two colour relations:
\begin{itemize}
\item the two colour associativity
\begin{equation}  \label{assoc3}
\centering
\begin{tikzpicture} [scale = 0.7]
\draw[color=blue, line width=0.8mm] (-1.45,0) -- (0,0);

\draw[color=red, line width=0.8mm] (.55,0) -- (1.25,0.75);
\draw[color=red, line width=0.8mm] (.55,0) -- (1.25,-0.75);

\draw[color=blue, line width=0.8mm] (.75,.75) -- (0,0);

\draw[color=blue, line width=0.8mm] (0,0) -- (.75,-0.75);

\draw[color=red, line width=0.8mm] (0,0) -- (-.75,-0.75);

\draw[color=red, line width=0.8mm] (-.75,.75) -- (0,0);

\draw[color=red, line width=0.8mm] (0,0) -- (.55,0);

\node at (2,0) {=};

\draw[color=red, line width=0.8mm] (3.85,0.35) -- (4.375,0.35);

\draw[color=red, line width=0.8mm] (4.35,0.35) -- (4.7,0.7);

\draw[color=red, line width=0.8mm] (3.85,-0.35) -- (4.375, -0.35);
\draw[color=red, line width=0.8mm] (4.35,-0.35) -- (4.7, -0.7);
\draw[color=red, line width=0.8mm] (3.35,1) -- (3.85,.35);
\draw[color=red, line width=0.8mm] (3.85,.35) -- (3.55,-.02);
\draw[color=red, line width=0.8mm] (3.35,-1) -- (3.85,-.35);
\draw[color=red, line width=0.8mm] (3.85,-.35) -- (3.55,0.01);

\draw[color=blue, line width=0.8mm] (2.5,0) -- (3,0);
\draw[color=blue, line width=0.8mm] (3,0) -- (3.35,0.35);
\draw[color=blue, line width=0.8mm] (3.325,0.35) -- (3.85,0.35);
\draw[color=blue, line width=0.8mm] (3,0) -- (3.35, -0.35);
\draw[color=blue, line width=0.8mm] (3.325,-0.35) -- (3.85, -0.35);
\draw[color=blue, line width=0.8mm] (3.85,.35) -- (4.15,-.02);
\draw[color=blue, line width=0.8mm] (4.35,1) -- (3.85,.35);
\draw[color=blue, line width=0.8mm] (3.85,-.35) -- (4.15,0.01);
\draw[color=blue, line width=0.8mm] (4.35,-1) -- (3.85,-.35);

\node at (7,0) {\text{if $m_{\textcolor{red}{s} \textcolor{blue}{t}} = 3, $ }};

\end{tikzpicture}
\end{equation}
\begin{equation}  \label{assoc4}
\centering
\begin{tikzpicture} [scale = 0.7]
\draw[color=red, line width=0.8mm] (-1.45,0) -- (.55,0);
\draw[color=red, line width=0.8mm] (.55,0) -- (1.25,0.75);
\draw[color=red, line width=0.8mm] (0,0.75) -- (0,-0.75);
\draw[color=red, line width=0.8mm] (.55,0) -- (1.25,-0.75);
\draw[color=blue, line width=0.8mm] (.75,.75) -- (-.75,-0.75);
\draw[color=blue, line width=0.8mm] (-.75,.75) -- (.75,-0.75);
\node at (2,0) {=};

\draw[color=red, line width=0.8mm] (2.5,0) -- (3,0);
\draw[color=red, line width=0.8mm] (3,0) -- (3.35,0.35);
\draw[color=red, line width=0.8mm] (3.325,0.35) -- (4.375,0.35);
\draw[color=red, line width=0.8mm] (4.35,0.35) -- (4.7,0.7);

\draw[color=red, line width=0.8mm] (3,0) -- (3.35, -0.35);
\draw[color=red, line width=0.8mm] (3.325,-0.35) -- (4.375, -0.35);
\draw[color=red, line width=0.8mm] (4.35,-0.35) -- (4.7, -0.7);

\draw[color=red, line width=0.8mm] (3.85,1) -- (3.85,-1);
\draw[color=blue, line width=0.8mm] (3.35,1) -- (4.15,-.02);
\draw[color=blue, line width=0.8mm] (4.35,1) -- (3.55,-.02);
\draw[color=blue, line width=0.8mm] (3.35,-1) -- (4.15,0.01);
\draw[color=blue, line width=0.8mm] (4.35,-1) -- (3.55,0.01);

\node at (7,0) {\text{if $m_{\textcolor{red}{s} \textcolor{blue}{t}} = 4 $};};

\end{tikzpicture}
\end{equation}

\item the dot-crossing relations

\end{itemize}
\begin{equation} \label{mst2}
\begin{tikzpicture} 
\draw[color=red, line width=0.8mm] (-1.45,0) -- (.55,0);

\draw[color=blue, line width=0.8mm] (0,0.75) -- (0,-0.75);

\filldraw[red]  (0.55,0) circle (3pt) ;

\node at (2,0) {=};

\draw[color=red, line width=0.8mm] (2.5,0) -- (3.5,0);
\filldraw[red]  (3.5,0) circle (3pt) ;

\filldraw[blue]  (4.3,0) circle (3pt) ;
\draw[color=blue, line width=0.8mm] (4.3,0) -- (4.85,0);
\draw[color=blue, line width=0.8mm] (4.85,-.75) -- (4.85,.75);

\draw[color=red, line width=0.8mm] (2.5,0) -- (3,0);

\node at (7,0) {\text{if $m_{\textcolor{red}{s} \textcolor{blue}{t}} = 2 $},};

\end{tikzpicture}
\end{equation}
\begin{equation} \label{mst3}
\centering
\begin{tikzpicture} 
\draw[color=blue, line width=0.8mm] (-1.45,0) -- (0,0);

\filldraw[red]  (0.55,0) circle (3pt) ;

\draw[color=blue, line width=0.8mm] (.75,.75) -- (0,0);

\draw[color=blue, line width=0.8mm] (0,0) -- (.75,-0.75);

\draw[color=red, line width=0.8mm] (0,0) -- (-.75,-0.75);

\draw[color=red, line width=0.8mm] (-.75,.75) -- (0,0);

\draw[color=red, line width=0.8mm] (0,0) -- (.55,0);

\node at (2,0) {=};

\draw[color=red, line width=0.8mm] (3.825,.45) -- (3.825,.75);
\draw[color=red, line width=0.8mm] (3.825,-.45) -- (3.825,-.75);

\draw[color=blue, line width=0.8mm] (4.65,0) -- (4.85,0);
\draw[color=blue, line width=0.8mm] (4.85,-.75) -- (4.85,.75);

\draw[color=blue, line width=0.8mm] (2.5,0) -- (3,0);
\filldraw[fill=white, draw=black,rounded corners] (3,-.45) rectangle (4.65,.45);

\node at (3.85,0) {$JW_2$}; 

\node at (7,0) {\text{if $m_{\textcolor{red}{s} \textcolor{blue}{t}} = 3 $},};

\end{tikzpicture}
\end{equation}
\begin{equation} \label{mst4}
\begin{tikzpicture} 
\draw[color=red, line width=0.8mm] (-1.45,0) -- (.55,0);

\draw[color=red, line width=0.8mm] (0,0.75) -- (0,-0.75);

\filldraw[red]  (0.55,0) circle (3pt) ;
\draw[color=blue, line width=0.8mm] (.75,.75) -- (-.75,-0.75);
\draw[color=blue, line width=0.8mm] (-.75,.75) -- (.75,-0.75);
\node at (2,0) {=};

\draw[color=red, line width=0.8mm] (2.5,0) -- (3,0);

\draw[color=blue, line width=0.8mm] (3.55,.45) -- (3.55,.75);
\draw[color=blue, line width=0.8mm] (3.55,-.45) -- (3.55,-.75);

\draw[color=red, line width=0.8mm] (4.1,.45) -- (4.1,.75);
\draw[color=red, line width=0.8mm] (4.1,-.45) -- (4.1,-.75);

\draw[color=blue, line width=0.8mm] (4.6,0) -- (4.85,0);
\draw[color=blue, line width=0.8mm] (4.85,-.75) -- (4.85,.75);

\draw[color=red, line width=0.8mm] (2.5,0) -- (3,0);
\filldraw[fill=white, draw=black,rounded corners] (3,-.45) rectangle (4.65,.45);

\node at (3.85,0) {{ $JW_3$}};

\node at (7,0) {\text{if $m_{\textcolor{red}{s} \textcolor{blue}{t}} = 4 $},};

\end{tikzpicture}
\end{equation}
where
\begin{align*}
  JW_2 &= 
\begin{tikzcd}
\draw[dashed,color=black!60] (0,0) circle (0.7);
\draw[blue, line width=0.8mm] (0,.7) -- (0,-.7);
\draw[red, line width=0.8mm] (-.2,0) -- (-.7,0);
\draw[red, line width=0.8mm] (.2,0) -- (.7,0);
\filldraw[red] (-.3,0) circle (3pt) ;
\filldraw[red] (.3,0) circle (3pt) ;
\end{tikzcd} 
- \frac{1}{a_{t,s}}
\begin{tikzcd}
\draw[dashed,color=black!60] (0,0) circle (0.7);
\draw[red, line width=0.8mm] (.7,0) -- (-.7,0);
\draw[blue, line width=0.8mm] (0,-.2) -- (0,-.7);
\draw[blue, line width=0.8mm] (0,.2) -- (0,.7);
\filldraw[blue] (0,-.3) circle (3pt) ;
\filldraw[blue] (0,.3) circle (3pt) ;
\end{tikzcd} \ ,
\\
JW_3 &= 
\begin{tikzcd}
\draw[dashed,color=black!60] (0,0) circle (0.7);
\draw[red, line width=0.8mm] (.125,.68) -- (.125,-.68);
\draw[blue, line width=0.8mm] (-.125,.68) -- (-.125,-.68);
\draw[red, line width=0.8mm] (-.5,0) -- (-.7,0);
\draw[blue, line width=0.8mm] (.45,0) -- (.7,0);
\filldraw[red] (-.425,0) circle (3pt) ;
\filldraw[blue] (.425,0) circle (3pt) ;
\end{tikzcd} 
- \frac{a_{t,s}}{a_{s,t}a_{t,s}-1} 
\begin{tikzcd}
\draw[dashed,color=black!60] (0,0) circle (0.7);
\draw[red, line width=0.8mm] (-.5,0) -- (-.7,0);
\draw[red, line width=0.8mm] (.23,0.4) -- (.23,0.66);
\draw[red, line width=0.8mm] (.23,-0.4) -- (.23,-0.66);
\draw[blue, line width=0.8mm] (0,0) -- (-.45,.58);
\draw[blue, line width=0.8mm] (0,0) -- (-.45,-.58);
\draw[blue, line width=0.8mm] (0,0) -- (.7,0);
\filldraw[red] (-.4,0) circle (3pt) ;
\filldraw[red] (.23,0.3) circle (3pt) ;
\filldraw[red] (.23,-0.3) circle (3pt) ;
\end{tikzcd} 
- \frac{a_{s,t}}{a_{s,t}a_{t,s}-1} 
\begin{tikzcd}
\draw[dashed,color=black!60] (0,0) circle (0.7);
\draw[blue, line width=0.8mm] (.5,0) -- (.7,0);
\draw[blue, line width=0.8mm] (-.23,0.4) -- (-.23,0.66);
\draw[blue, line width=0.8mm] (-.23,-0.4) -- (-.23,-0.66);
\draw[red, line width=0.8mm] (0,0) -- (.45,.58);
\draw[red, line width=0.8mm] (0,0) -- (.45,-.58);
\draw[red, line width=0.8mm] (0,0) -- (-.7,0);
\filldraw[blue] (.4,0) circle (3pt) ;
\filldraw[blue] (-.23,0.3) circle (3pt) ;
\filldraw[blue] (-.23,-0.3) circle (3pt) ;
\end{tikzcd} 
+ \frac{1}{a_{s,t}a_{t,s}-1} 
\begin{tikzcd}
\draw[dashed,color=black!60] (0,0) circle (0.7);
\draw[blue, line width=0.8mm] (.13,0) -- (.7,0);
\draw[blue, line width=0.8mm] (.13,0) -- (-.45,.58);
\draw[red, line width=0.8mm] (-.15,0) -- (.45,-.58);
\draw[red, line width=0.8mm] (-.15,0) -- (-.7,0);
\draw[red, line width=0.8mm] (.23,0.4) -- (.25,0.66);
\draw[blue, line width=0.8mm] (-.23,-0.4) -- (-.25,-0.66);
\filldraw[red] (.23,0.3) circle (3pt) ;
\filldraw[blue] (-.23,-0.3) circle (3pt) ;
\end{tikzcd} \\
&+ \frac{1}{a_{s,t}a_{t,s}-1}
\begin{tikzcd}
\draw[dashed,color=black!60] (0,0) circle (0.7);
\draw[blue, line width=0.8mm] (.13,0) -- (.7,0);
\draw[red, line width=0.8mm] (.23,-0.4) -- (.23,-0.66);
\draw[blue, line width=0.8mm] (-.23,0.4) -- (-.23,0.66);
\draw[red, line width=0.8mm] (-.13,0) -- (.45,.58);
\draw[red, line width=0.8mm] (-.13,0) -- (-.7,0);
\draw[blue, line width=0.8mm] (.13,0) -- (-.45,-.58);
\filldraw[red] (.23,-0.3) circle (3pt) ;
\filldraw[blue] (-.23,0.3) circle (3pt) ;
\end{tikzcd} \ ; \\	
\end{align*}  
\item The three colour relations or ``Zamolodzhikov'' relations:
\begin{itemize}
\item A sub-Coxeter system of type $A_1 \times I_2(m),$ that is, the Coxeter graph of the parabolic subgroup genetated by $( $\textcolor{red}{s}$,$\textcolor{blue}{t}$,$\textcolor{green}{u}$)$ is 
\begin{tikzcd}
\filldraw[black]  (0,0) circle (2pt) ;
\filldraw[black]  (0.7,0) circle (2pt) ;
\filldraw[black]  (1.4,0) circle (2pt) ;
\draw [line width = 0.2mm] (0,0) -- (0.7,0);
\draw  (0.35,.2) node {{m}};
\draw (0,-.35) node {{\textcolor{red}{s}}};
\draw  (0.7,-.35) node {{\textcolor{blue}{t}}};
\draw  (1.4,-.35) node {{\textcolor{green}{u}}};
\end{tikzcd}
\begin{equation}    \label{Zamo4}
\centering
\begin{tikzpicture} [scale = 0.7]

\draw[color=red, line width=0.8mm] (0,0) -- (.3,0.9);
\draw[color=red, line width=0.8mm] (0,0) -- (-.55,0.75);
\draw[color=red, line width=0.8mm] (0,0) -- (.65,-0.75);
\draw[color=red, line width=0.8mm] (-.2,-.9) -- (0,0);

\draw[color=blue, line width=0.8mm] (0,0) -- (-.55,-0.75);
\draw[color=blue, line width=0.8mm] (0,0) -- (.3,-0.9);
\draw[color=blue, line width=0.8mm] (0,0) -- (-.2,.9);
\draw[color=blue, line width=0.8mm] (.65,.75) -- (0,0);
\draw[color=green!200, line width=0.8mm]  (.9, .3) .. controls (-.1, .7) and (.1,.7) .. (-.8, .3) ;

\node at (2,0) {=};

\draw[color=red, line width=0.8mm] (4,0) -- (4.3,0.9);
\draw[color=red, line width=0.8mm] (4,0) -- (3.45,0.75);
\draw[color=red, line width=0.8mm] (4,0) -- (4.65,-0.75);
\draw[color=red, line width=0.8mm] (3.8,-.9) -- (4,0);

\draw[color=blue, line width=0.8mm] (4,0) -- (3.45,-0.75);
\draw[color=blue, line width=0.8mm] (4,0) -- (4.3,-0.9);
\draw[color=blue, line width=0.8mm] (4,0) -- (3.8,.9);
\draw[color=blue, line width=0.8mm] (4.65,.75) -- (4,0);
\draw[color=green!200, line width=0.8mm]  (4.9, -.3) .. controls (3.9, -.7) and (4.1,-.7) .. (3.2, -.3) ;

\node at (8,0) {\text{if $m_{\textcolor{red}{s} \textcolor{blue}{t}} = 4 $},};
\node at (-4,0) {$(B_2 \times A_1)$};
\end{tikzpicture}
\end{equation}
\begin{equation}  \label{Zamo3}
\centering
\begin{tikzpicture} [scale = 0.7]

\draw[color=red, line width=0.8mm] (0,0) -- (0,.85);
\draw[color=red, line width=0.8mm] (0,0) -- (-.5,-0.75);
\draw[color=red, line width=0.8mm] (0,0) -- (.5,-0.75);

\draw[color=blue, line width=0.8mm] (0,0) -- (0,-.85);
\draw[color=blue, line width=0.8mm] (0,0) -- (-.5,0.75);
\draw[color=blue, line width=0.8mm] (0,0) -- (.5,0.75);

\draw[color=green!200, line width=0.8mm]  (.9, .3) .. controls (-.1, .7) and (.1,.7) .. (-.8, .3) ;

\node at (2,0) {=};

\draw[color=red, line width=0.8mm] (4,0) -- (4,.85);
\draw[color=red, line width=0.8mm] (4,0) -- (3.5,-0.75);
\draw[color=red, line width=0.8mm] (4,0) -- (4.5,-0.75);

\draw[color=blue, line width=0.8mm] (4,0) -- (4,-.85);
\draw[color=blue, line width=0.8mm] (4,0) -- (3.5,0.75);
\draw[color=blue, line width=0.8mm] (4,0) -- (4.5,0.75);
\draw[color=green!200, line width=0.8mm]  (4.9, -.3) .. controls (3.9, -.7) and (4.1,-.7) .. (3.2, -.3) ;

\node at (8,0) {\text{if $m_{\textcolor{red}{s} \textcolor{blue}{t}} = 3 $},};
\node at (-4,0) {$(A_2 \times A_1)$};
\end{tikzpicture}
\end{equation}
\begin{equation}  \label{Zamo2}
\centering
\begin{tikzpicture} [scale = 0.7]

\draw[color=red, line width=0.8mm] (-.5,-.75) -- (0.5,.75);
\draw[color=blue, line width=0.8mm] (.5,-.75) -- (-0.5,.75);

\draw[color=green!200, line width=0.8mm]  (.9, .3) .. controls (-.1, .7) and (.1,.7) .. (-.8, .3) ;

\node at (2,0) {=};

\draw[color=red, line width=0.8mm] (3.5,-.75) -- (4.5,.75);
\draw[color=blue, line width=0.8mm] (4.5,-.75) -- (3.5,.75);

\draw[color=green!200, line width=0.8mm]  (4.9, -.3) .. controls (3.9, -.7) and (4.1,-.7) .. (3.2, -.3) ;

\node at (8,0) {\text{if $m_{\textcolor{red}{s} \textcolor{blue}{t}} =  2 $},};

\node at (-3.5,0) {$(A_1 \times A_1 \times A_1)$};
\end{tikzpicture}
\end{equation}
\begin{equation} \label{A3}
\includegraphics[scale=1]{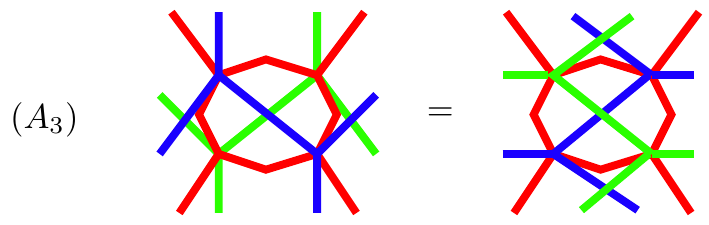} \text{,}
\end{equation} 
\begin{equation} \label{B3}
\includegraphics[scale=1]{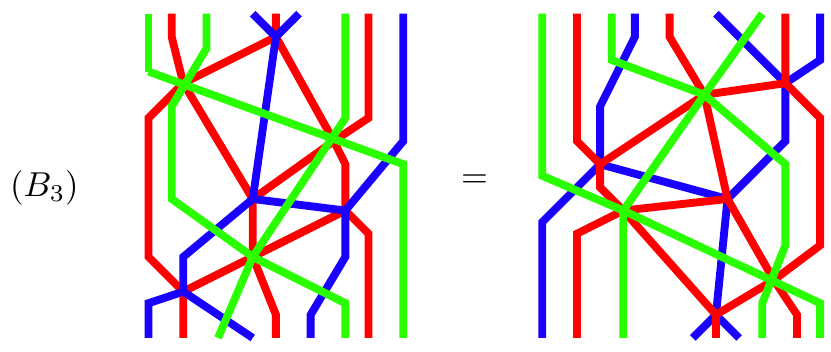}.
\end{equation}
\end{itemize}
\end{enumerate} 
 \end{enumerate} 
 
This concludes the definition of $\cD.$ 
 
 Note that the one colour relations from \eqref{Frobenius} and \eqref{wall} combined with isotopy invariance encode the datum of the objects $\color{purple} s$ being Frobenius objects\footnote{Note that they are actually nilpotent, as they also satisfy the relation: 
    \begin{tikzpicture}[scale=0.3, baseline]
        \draw[purple, line width=0.5mm] (0,.35) -- (0,1);
        \draw[purple, line width=0.5mm] (0,0) circle (.35);
        \draw[purple, line width=0.5mm] (0,-.35) -- (0,-1);
        \node at (1.5,.2) {= 0 };
    \end{tikzpicture} 
(implied by other relations).}.
 For completeness, we shall spell out of the corresponding one colour relations that comes with it.
 These relations are expressed in term of cups and caps defined as follows:
 \begin{center}
 \begin{tikzcd}
 \draw[purple, line width=0.8mm] (-.6,-0.5) arc(180 : 0: 0.3);
 \draw[purple, line width=0.8mm] (-.6,-0.5) -- (-.6, -.7);
  \draw[purple, line width=0.8mm] (0,-0.5) -- (0, -.7);

\node at (.5, -0.5) {:=};
\draw[ purple, line width=0.8mm] (1.5,0) -- (1.5,-.45);
\draw[ purple, line width=0.8mm] (1.5,-.4) .. controls (1.3,-.5) and (1.25, -.55)..  (1.2,-0.8);
\draw[ purple, line width=0.8mm] (1.5,-.4) .. controls (1.7,-.5) and (1.75, -.55)..  (1.8,-0.8);
\filldraw[purple] (1.5,0) circle (3pt) ;

 \node at (-3,-.5) {\text{The definition of cap}};
  \end{tikzcd} 
  \end{center}
  \begin{center}
   \begin{tikzcd}
 \draw[purple, line width=0.8mm] (-.6,0.5) arc(-180 : -0: 0.3);
 \draw[purple, line width=0.8mm] (-.6,0.5) -- (-.6, .8);
  \draw[purple, line width=0.8mm] (0,0.5) -- (0, .8);

\node at (.5, 0.5) {:=};
\draw[ purple, line width=0.8mm] (1.5,0) -- (1.5,.45);
\draw[ purple, line width=0.8mm] (1.5,.4) .. controls (1.3,.5) and (1.25, .55)..  (1.2,0.8);
\draw[ purple, line width=0.8mm] (1.5,.4) .. controls (1.7,.5) and (1.75, .55)..  (1.8,0.8);
\filldraw[purple] (1.5,0) circle (3pt) ;

 \node at (-3,.5) {\text{The definition of cup}};
 
  \end{tikzcd}.
  \end{center}
 \noindent The trivalent vertices give multiplication or comultiplication while the dots provide unit or counit.
   To obtain \eqref{Frobenius}, we can rotate \eqref{assocmult} (resp. \eqref{coassoccomult}) using a cup (resp. a cap) and apply \eqref{comult rot mult} (resp. \eqref{mult rot comult}).
   On the other hand, \eqref{mult rot comult} to \eqref{unit rot counit} records the cyclicity of all morphisms. 
   Therefore, \eqref{wall} can be replaced by \eqref{enddot counit comult} and \eqref{startdot unit multi}.
Here is the list of them:
 \begin{equation} \label{assocmult} \begin{tikzcd}
\draw[ purple, line width=0.8mm] (1.5,1.2) -- (1.5,.65);
\draw[ purple, line width=0.8mm] (1.5,.7) .. controls (1.2,.6) and (1, .3)..  (.9,-.4);
\draw[ purple, line width=0.8mm] (1.5,.7) .. controls (1.7,.55) and (1.75, .45)..  (1.8,.2);

\draw[ purple, line width=0.8mm] (1.8,.2) .. controls (1.45,0) and (1.5, -.55)..  (1.45,-.4);
\draw[ purple, line width=0.8mm] (1.8,.2) .. controls (2.15,0) and (2.1, -.55)..  (2.15,-.4);

\node at (2.55, 0.55) {=};

\draw[ purple, line width=0.8mm] (3.6,1.2) -- (3.6,.65);
\draw[ purple, line width=0.8mm] (3.6,.7) .. controls (3.9,.6) and (4.1, .3)..  (4.2,-.4);
\draw[ purple, line width=0.8mm] (3.6,.7) .. controls (3.4,.55) and (3.35, .45)..  (3.3,.2);

\draw[ purple, line width=0.8mm] (3.3,.2) .. controls (3.65,0) and (3.6, -.55)..  (3.65,-.4);
\draw[ purple, line width=0.8mm] (3.3,.2) .. controls (2.95,0) and (2.9, -.55)..  (2.95,-.4);

 \node at (8,.55) {\text{ the associativity of the multiplication}.};
 
  \end{tikzcd}
\end{equation}

 \begin{equation}  \label{coassoccomult}
 \begin{tikzcd}
\draw[ purple, line width=0.8mm] (1.5,-0.1) -- (1.5,.45);
\draw[ purple, line width=0.8mm] (1.5,.4) .. controls (1.2,.5) and (1, .8)..  (.9,1.5);
\draw[ purple, line width=0.8mm] (1.5,.4) .. controls (1.7,.55) and (1.75, .65)..  (1.8,.9);

\draw[ purple, line width=0.8mm] (1.8,.9) .. controls (1.45,1.1) and (1.5, 1.45)..  (1.45,1.5);
\draw[ purple, line width=0.8mm] (1.8,.9) .. controls (2.15,1.1) and (2.1, 1.45)..  (2.15,1.5);

\node at (2.55, 0.55) {=};

\draw[ purple, line width=0.8mm] (3.6,-0.1) -- (3.6,.45);
\draw[ purple, line width=0.8mm] (3.6,.4) .. controls (3.9,.5) and (4.1, .8)..  (4.2,1.5);
\draw[ purple, line width=0.8mm] (3.6,.4) .. controls (3.4,.55) and (3.35, .65)..  (3.3,.9);

\draw[ purple, line width=0.8mm] (3.3,.9) .. controls (3.65,1.1) and (3.6, 1.45)..  (3.65,1.5);
\draw[ purple, line width=0.8mm] (3.3,.9) .. controls (2.95,1.1) and (2.9, 1.45)..  (2.95,1.5);

 \node at (8,.55) {\text{ the coassociativity of the comultiplication.}};
 
  \end{tikzcd}
\end{equation} 

 \begin{equation} \label{enddot counit comult}\begin{tikzcd}
\draw[ purple, line width=0.8mm] (0,0) -- (0,.65);
\draw[ purple, line width=0.8mm] (0,.6) .. controls (-.2,.7) and (-.25, .75)..  (-.3,1.1);
\draw[ purple, line width=0.8mm] (0,.6) .. controls (.2,.7) and (.25, .75)..  (.25,.75);

\node at (.9, 0.4) {=};
\draw[ purple, line width=0.8mm] (2.8,0) -- (2.8,.65);
\draw[ purple, line width=0.8mm] (2.8,.6) .. controls (2.6,.7) and (2.55, .75)..  (2.55,0.75);
\draw[ purple, line width=0.8mm] (2.8,.6) .. controls (3.0,.7) and (3.05, .75)..  (3.1,1.1);

\draw[ purple, line width=0.8mm] (1.4, 0) -- (1.4, 1);

\filldraw[purple] (.25,.75) circle (3pt) ;

\filldraw[purple] (2.55,0.75) circle (3pt) ;

\node at (2, 0.4) {=};

 \node at (7.5,.5) {\text{the enddot as counit for the comultiplication.}};
 
  \end{tikzcd}
\end{equation} 

\begin{equation} \label{startdot unit multi}    
\begin{tikzcd}
\draw[ purple, line width=0.8mm] (0,1) -- (0,.35);
\draw[ purple, line width=0.8mm] (0,.4) .. controls (-.2,.3) and (-.25, .25)..  (-.3,-0.1);
\draw[ purple, line width=0.8mm] (0,.4) .. controls (.2,.3) and (.25, .25)..  (.25,.25);

\node at (.9, 0.4) {=};
\draw[ purple, line width=0.8mm] (2.8,1) -- (2.8,.35);
\draw[ purple, line width=0.8mm] (2.8,.4) .. controls (2.6,.3) and (2.55, .25)..  (2.55,0.25);
\draw[ purple, line width=0.8mm] (2.8,.4) .. controls (3.0,.3) and (3.05, .25)..  (3.1,-0.1);

\draw[ purple, line width=0.8mm] (1.4, 0) -- (1.4, 1);

\filldraw[purple] (.25,.25) circle (3pt) ;

\filldraw[purple] (2.55,0.25) circle (3pt) ;

\node at (2, 0.4) {=};

 \node at (7.5,.5) {\text{the startdot as unit for the multiplication.}};
 
  \end{tikzcd}
\end{equation} 
 \begin{equation} \label{selfadjoint}    \begin{tikzcd}
\draw[purple, line width=0.8mm] (-.3,0.5) arc(180 : 0: 0.25);
\draw[ purple, line width=0.8mm] (.2,.5) .. controls (0.1,-.1) and (-.2, 0)..  (-.3,-.5);
\draw[purple, line width=0.8mm] (-.3,0.5) arc(0 : -180: 0.25);
\draw[ purple, line width=0.8mm] (-.8,.5) .. controls (-.7,1.1) and (-.4, 1)..  (-.3,1.5);

\node at (.75, 0.4) {=};
\draw[ purple, line width=0.8mm] (1.5,1.5) -- (1.5,-.5);

\node at (2.1, 0.4) {=};

\draw[purple, line width=0.8mm] (3.2,0.5) arc(-180 : 0: 0.25);
\draw[ purple, line width=0.8mm] (3.7,.5) .. controls (3.6,1.1) and (3.3, 1)..  (3.2,1.5);
\draw[purple, line width=0.8mm] (3.2,0.5) arc(0 : 180: 0.25);
\draw[ purple, line width=0.8mm] (2.7,.5) .. controls (2.8,-.1) and (3.1, 0)..  (3.2,-.5);

 \node at (7,.5) {\text{the self-biadjointness of $B_s$.}};
 
  \end{tikzcd}
\end{equation} 
 \begin{equation} \label{mult rot comult} 
   \begin{tikzcd}
\draw[ purple, line width=0.8mm] (-.3,0) -- (-.3,0.45);
\draw[ purple, line width=0.8mm] (-.3,.4) .. controls (-.5,.5) and (-.55, .55)..  (-.6,0.8);
\draw[ purple, line width=0.8mm] (-0.3,.4) .. controls (-.1,.7) and (0.17, .7)..  (0.2,0);

\node at (.6, 0.4) {=};
\draw[ purple, line width=0.8mm] (1.5,0.8) -- (1.5,.35);
\draw[ purple, line width=0.8mm] (1.5,.4) .. controls (1.3,.3) and (1.25, .25)..  (1.2,0);
\draw[ purple, line width=0.8mm] (1.5,.4) .. controls (1.7,.3) and (1.75, .25)..  (1.8,0);

\node at (2.3, 0.4) {=};

\draw[ purple, line width=0.8mm] (3.2,0) -- (3.2,0.45);
\draw[ purple, line width=0.8mm] (3.2,.4) .. controls (3.4,.5) and (3.45, .55)..  (3.5,0.8);
\draw[ purple, line width=0.8mm] (3.2,.4) .. controls (3,.7) and (2.8, .7)..  (2.7,0);

 \node at (8,.5) {\text{the mult. is a ``half rotation'' of the comult.}};
 
  \end{tikzcd}
\end{equation} 
\begin{equation}  \label{comult rot mult}
   \begin{tikzcd}
\draw[ purple, line width=0.8mm] (-.3,0.8) -- (-.3,0.35);
\draw[ purple, line width=0.8mm] (-.3,.4) .. controls (-.5,.3) and (-.55, .25)..  (-.6,0);
\draw[ purple, line width=0.8mm] (-0.3,.4) .. controls (-.1,.1) and (0.17, .1)..  (0.2,0.8);

\node at (.6, 0.4) {=};
\draw[ purple, line width=0.8mm] (1.5,0) -- (1.5,.45);
\draw[ purple, line width=0.8mm] (1.5,.4) .. controls (1.3,.5) and (1.25, .55)..  (1.2,0.8);
\draw[ purple, line width=0.8mm] (1.5,.4) .. controls (1.7,.5) and (1.75, .55)..  (1.8,0.8);

\node at (2.3, 0.4) {=};

\draw[ purple, line width=0.8mm] (3.2,0.8) -- (3.2,0.35);
\draw[ purple, line width=0.8mm] (3.2,.4) .. controls (3.4,.3) and (3.45, .25)..  (3.5,0);
\draw[ purple, line width=0.8mm] (3.2,.4) .. controls (3,.1) and (2.8, .1)..  (2.7,0.8);

 \node at (8,.5) {\text{the comult. is a half ``rotation'' of the mult.}};
 
  \end{tikzcd}
\end{equation} 
 \begin{equation} \label{counit rot unit}
\begin{tikzcd}
 \draw[purple, line width=0.8mm] (-.6,0.5) arc(180 : 0: 0.3);
  \draw[purple, line width=0.8mm] (0,0.5) -- (0, 0);
\filldraw[purple] (-.6,0.45) circle (3pt) ;

 \draw[purple, line width=0.8mm] (2,0.5) arc(180 : 0: 0.3);
 \draw[purple, line width=0.8mm] (2,0.5) -- (2, 0);
\filldraw[purple] (2.6,0.45) circle (3pt) ;

\node at (.5, 0.45) {=};
\draw[ purple, line width=0.8mm] (1,0) -- (1,.9);
\filldraw[purple] (1,0.9) circle (3pt) ;

\node at (1.5, 0.5) {=};

 \node at (7,.5) {\text{the counit is a ``rotation'' of the unit.}};
 \end{tikzcd}
\end{equation}  
\begin{equation} \label{unit rot counit}
\begin{tikzcd}
 \draw[purple, line width=0.8mm] (-.6,0.5) arc(-180 : -0: 0.3);
  \draw[purple, line width=0.8mm] (0,0.5) -- (0, 1);
  \draw[purple, line width=0.8mm] (-.6,0.5) -- (-.6, .6);
\filldraw[purple] (-.6,0.55) circle (3pt) ;
 \draw[purple, line width=0.8mm] (2,0.5) arc(-180 : -0: 0.3);
 \draw[purple, line width=0.8mm] (2,0.5) -- (2, 1);
\filldraw[purple] (2.6,0.55) circle (3pt) ;
\node at (.5, 0.5) {=};
\draw[ purple, line width=0.8mm] (1,0.2) -- (1,1);
\filldraw[purple] (1,0.2) circle (3pt) ;
\node at (1.5, 0.5) {=};
 \node at (7,.5) {\text{the unit is a ``rotation'' of the counit.}};
 \end{tikzcd}
\end{equation} 

This marks the end of the supplementary one colour relations.

It is a theorem of Elias--Williamson \cite[Theorem 6.28]{EW_2016} that the Karoubi envelope $\cK ar(\overline{\cD})$ of the graded additive closure $\overline{\cD}$ of $\cD$
  is equivalent to the category of Soergel bimodules as additive, monoidal categories.
  In particular, $\cK ar(\overline{\cD})$ categorifies the type $B_n$ Hecke algebra.

\subsection{The 2-Braid Group}
The elementary Rouquier complexes corresponding to a simple reflection $s \in S$ are defined as follows (cf. \cite{rouquier_2006}):
\begin{align*}
F_s &:= 0 \ra B_s \xra{\begin{tikzcd}
\draw[ purple, line width=0.7mm] (0,0) -- (0,-.35);
\filldraw[purple] (0,0) circle (3pt) ;
\end{tikzcd} } R(1) \ra 0; \\
E_s = F_{s^{-1}} &:= 0 \ra R(-1) \xra{\begin{tikzcd}
\draw[ purple, line width=0.7mm] (0,0) -- (0,-.35);
\filldraw[purple] (0,-.35) circle (3pt) ;
\end{tikzcd} } B_s \ra 0,
\end{align*}
with both $B_s$ in cohomological degree $0.$
The \emph{$2$-braid group} 2-$\cB r$ of type $B_n$ is the full monoidal subcategory of $\Kom^b(\cK ar(\overline{\cD}))$ generated by $F_s$ and $E_s$ for all $s \in S$.
Observe that the set of isomorphism classes of objects in $2$-$\cB r,$ denoted by Pic$(2$-$\cB r)$, forms a group called the \textit{Picard group} of the monoidal category $2$-$\cB r$  under tensor product composition.
Rouquier showed that $F_s$ and $E_s$ are inverses of each other, and moreover they satisfy the required braid relations \cite[Proposition 9.2, Lemma 9.3]{rouquier_2006}.
In particular, the map sending $\sigma_j$ to the isomorphism class $F_{s_j} \in $ Pic$(2$-$\cB r)$ for all $j \geq 1$ is well-defined \cite[Proposition 9.4]{rouquier_2006}.
Rouquier conjectures that this assignment is moreover faithful in general, and we shall prove this in the type $B$ case.

\section{Quotient category and proof of faithfulness}\label{quotient and proof}
The aim of this section is to show that $\cK ar(\ol{\cB})$ is a quotient category of $\cK ar(\ol{\cD})$.
This will be given by an essentially surjective monoidal functor $G_0: \cD \ra \cB$, which leads to an essentially surjective additive, monoidal functor on the Karoubi envelope of their respective additive closures.
For clarity, we will use \textcolor{red}{red} or \textcolor{violet}{violet} for a simple reflection $j$, \textcolor{blue}{blue} for an adjacent simple reflection $j \pm 1,$ and \textcolor{green!100}{green} for a distant simple reflection $k$ with $|j-k|>1.$
    
We shall define $G_0$ on generating objects by sending $j$ in $S$ to $\cU_j$ and on generating morphisms as in the proof of \cref{maintheorem}, where the five types of $(\Ba_n, \Ba_n)$-bimodule homomorphisms needed are defined as follows:
\begin{itemize}
    \item $(-1)^{j+1} \alpha_j:  \cU_j \ra \cU_j(-1) \oplus \cU_j(1) \xra{\cong} \cU_j \otimes_{\Ba_n} \cU_j $ 
    is a morphism of degree (-1) defined by
    \[
    e_j \otimes e_j \mapsto (0, e_j \otimes e_j) \mapsto e_j \otimes e_j \otimes_{\Ba_n} e_j \otimes e_j,
    \]
    \item $\delta_j: \cU_j \otimes_{\Ba_n} \cU_j \xra{\cong} \cU_j(-1) \oplus \cU_j(1) \ra \cU_j$ 
    is a morphism  of degree (-1) defined by
    \begin{alignat*}{4}
        & e_j \otimes X_j \otimes_{\Ba_n} e_j \otimes e_j \
        && \mapsto \ 
        && (e_j \otimes e_j,0 ) \mapsto e_j \otimes e_j; \\
        & e_j \otimes e_j \otimes_{\Ba_n} e_j \otimes e_j  \
        && \mapsto \ 
        && (0, e_j \otimes e_j) \mapsto 0,
    \end{alignat*}
    \item $\epsilon_j : \Ba_n \ra \Ba_n$ is a morphism of degree (2) defined by
    \[
    1 \mapsto
    \begin{cases}
    (-1)^{i+1} (2X_j + 2X_{j+1}), &\text{for } j=1;\\
    (-1)^{i+1} (2X_j + X_{j-1} + X_{j+1}), &\text{for } 1<j < n; \\
    (-1)^{i+1} ( 2X_j + X_{j-1}), &\text{for } j = n,
    \end{cases}
    \]
    \item $(-1)^{j+1}\beta_j : \cU_j \ra \Ba_n$ is a morphism of degree (1) defined by
    \[
    e_j \otimes e_j \mapsto e_j,
    \]
    \item $\gamma_j  :  \Ba_n \ra  \cU_j$  is a morphism of degree (1) defined by
    \[
    1 \mapsto
        \begin{cases}
        X_j \otimes e_j + e_j \otimes X_j + (j+1|j) \otimes (j|j+1) \\
        \hspace{8mm} + (-ie_{j+1})(j+1|j) \otimes (j|j+1)(ie_{j+1}), &\text{for } j=1;\\
        X_j \otimes e_j + e_j \otimes X_j + (j-1|j) \otimes (j|j-1) + (j+1|j) \otimes (j|j+1), &\text{for } 1<j < n; \\
        X_j \otimes e_j + e_j \otimes X_j + (j-1|j) \otimes (j|j-1), &\text{for } j = n.
        \end{cases}
    \]
\end{itemize}
Note that the last two maps $\beta_j$ and $\gamma_j$ are exactly the bimodule maps in \cref{beta and gamma maps}. 

The following lemma will be crucial:
\begin{lemma}[{\cite[Lemma 3.11]{heng_nge}}] \label{P_i biadjoint}
We have the adjoint pairs $(P_j^B \otimes_{\mathbb{K}_j} - , \ {}_{j}{P}^B \otimes_{\Ba_n} -)$ and $( {}_{j}{P}^B (2) \otimes_{\Ba_n} - , \ P_j^B \otimes_{\mathbb{K}_j} - )$.
\end{lemma}

\begin{proposition} Denote $\mathbb{K}_j := \R$ when $j = 1$ and $\mathbb{K}_j := \C$ when $j \geq 2$.
We have the following identification of grading preserving $(\Ba_n, \Ba_n)$-bimodules morphism spaces:
\begin{enumerate}
\item $\cU_j \ra \Ba_n$ of degree $(1)$ is isomorphic to $\mathbb{K}_j \beta_j;$ 
    \label{eqn:bimodmapUB}
\item $\Ba_n \ra \cU_j$ of degree $(1)$ is isomorphic to $\mathbb{K}_j \gamma_j;$
\item $\cU_j \ra \cU_j \otimes_{\Ba_n} \cU_j $ of degree $(-1)$ is isomorphic to $\mathbb{K}_j \alpha_j;$
\item $\cU_j \otimes_{\Ba_n} \cU_j \ra \cU_j $ of degree $(-1)$ is isomorphic to $\mathbb{K}_j \delta_j;$
    \label{eqn:bimodmapUUU}
\item $\Ba_n \ra \Ba_n$ of degree $(2)$ is isomorphic to $\oplus_{1 \leq j \leq n} \mathbb{K}_jX_j,$ where $X_j$ is interpreted with left multiplication (equivalently right multiplication) by $X_j$. 
    \label{eqn:bimodmapBB}
\end{enumerate}
\end{proposition}
\begin{proof}
\eqref{eqn:bimodmapBB} follows directly from the fact that the map is completely determined by the image of $1 \in \Ba_n$ together with the degree restriction.

The morphisms $\beta_j, \gamma_j, \alpha_j$ and $\delta_j$ in \eqref{eqn:bimodmapUB} -- \eqref{eqn:bimodmapUUU} are indeed non-trivial $(\Ba_n,\Ba_n)$-bimodule morphisms of the respective spaces; it is therefore sufficient to show that hom spaces are all one-dimensional.
This follows from an easy computation using the adjoint pair $(P_j^B \otimes_{\mathbb{K}_j} - , \ {}_{j}{P}^B \otimes_{\Ba_n} -)$ in \cref{P_i biadjoint}.
For example, \eqref{eqn:bimodmapUUU} follows from the following identification of hom spaces:
\begin{align*}
    \Hom_{\Ba_n\text{-bimod}}(\cU_j \otimes_{\Ba_n} \cU_j, \cU_j(-1)) 
        &= \Hom_{\Ba_n\text{-bimod}}(P_j \otimes_{\mathbb{K}_j} {}_jP \otimes_{\Ba_n} P_j \otimes_{\mathbb{K}_j} {}_jP(2), P_j \otimes_{\mathbb{K}_j} {}_jP) \\
        &\cong \Hom_{\text{mod-}\Ba_n}( {}_jP \otimes_{\Ba_n} P_j \otimes_{\mathbb{K}_j} {}_jP(2), {}_jP \otimes_{\Ba_n} P_j \otimes_{\mathbb{K}_j} {}_jP) \\
        &\cong \Hom_{\text{mod-}\Ba_n}\left( {}_jP(2) \oplus {}_jP, {}_jP \oplus {}_jP(-2) \right).
\end{align*}
Since $\Hom_{\text{mod-}\Ba_n}\left( {}_jP, {}_jP (k)\right) \cong \mathbb{K}_j$ if and only if $k = 0$ or $2$, we have that 
\[
\Hom_{\Ba_n\text{-bimod}}(\cU_j \otimes_{\Ba_n} \cU_j, \cU_j(-1)) \cong \mathbb{K}_j
\]
as required.
We leave the other cases to the reader.
\end{proof}

\begin{theorem} \label{maintheorem}
There is an essentially surjective monoidal functor $G_0 : \cD \ra \cB$ which sends:
\begin{enumerate}[(i)]
\item the empty sequence $\emptyset$ in $S$ to  $\Ba_n,$ the monoidal identity in $\cB;$
\item $s_j$ in $S$ to $\cU_j;$ and
\item a sequence $s_{i_1} s_{i_2} \cdots s_{i_k}$ in $S$ of length $k \geq 2$ to $G_0( i_1 i_2 \cdots i_{k-1}) \otimes_{\Ba_n} G_0(i_k).$
\end{enumerate}
\end{theorem}

\begin{proof}
The assignment of $G_0 : \cD \ra \cB$ on objects is done as in the statement of the theorem, whereas on generating morphisms it is given in \cref{fig: quotient functor on morphism}.

\begin{figure}
\begin{tabular}{c c l}
\begin{tikzcd}
\draw[dashed,color=black!60] (0,0) circle (0.5);
\draw[ violet, line width=0.8mm] (0,0) -- (0,-.5);
\filldraw[violet] (0,0) circle (3pt) ;
\end{tikzcd} 
& $\longmapsto$ 
& $b_j \beta_j:\cU_j \ra \Ba_n   \text{ of degree 1}$; \\
\\
\begin{tikzcd}
\draw[dashed,color=black!60] (0,0) circle (0.5);
\draw[ violet, line width=0.8mm] (0,0) -- (0,.5);
\filldraw[violet] (0,0) circle (3pt) ;
\end{tikzcd}
& $\longmapsto$
& $ c_j \gamma_j:\Ba_n \ra \cU_j   \text{ of degree 1}$;
\\
\\
\begin{tikzcd}
\draw[dashed,color=black!60] (0,0) circle (0.5);
\draw[ violet, line width=0.8mm] (0,0) -- (0,-.5);
\draw[ violet, line width=0.8mm] (0,0) -- (-.35,.35);
\draw[ violet, line width=0.8mm] (0,0) -- (.35,.35);
\end{tikzcd}
&$\longmapsto$
& $a_j \alpha_j:\cU_j \ra \cU_j \otimes_{\Ba_n} \cU_j  \text{ of degree -1}$;
\\
\\
\begin{tikzcd}
\draw[dashed,color=black!60] (0,0) circle (0.5);
\draw[ violet, line width=0.8mm] (0,0) -- (0,.5);
\draw[ violet, line width=0.8mm] (0,0) -- (-.35,-.35);
\draw[ violet, line width=0.8mm] (0,0) -- (.35,-.35);
\end{tikzcd}
&$\longmapsto$
&$d_j \delta_j: \cU_j \otimes_{\Ba_n} \cU_j \ra \cU_j  \text{ of degree -1}$;
\\
\\
\begin{tikzcd}
\draw[dashed,color=black!60] (0,0) node{\textcolor{black}{\alpha}_{\textcolor{violet}{s_i}}} circle (0.5);
\end{tikzcd}
&$\longmapsto$
&$\epsilon_j:\Ba_n \ra \Ba_n, 1 \mapsto \sum\limits_{k=1}^n f^j_k X_k  \text{ of degree 2}$;
\\
\\
\begin{tikzcd}
\draw[dashed,color=black!60] (0,0) circle (0.5);

\draw[ violet, line width=0.8mm] (.35,.35) -- (-.35,-.35);
\draw[ green, line width=0.8mm] (-.35,.35) -- (.35,-.35);
\end{tikzcd}
&$\longmapsto$
&$0$ \text{  as $\cU_j \otimes_{\Ba_n} \cU_k = 0$, for $|j-k| > 1$};
\\
\\
\begin{tikzcd}
\draw[dashed,color=black!60] (0,0) circle (0.5);
\draw[ red, line width=0.8mm] (0,0) -- (0,.5);
\draw[ red, line width=0.8mm] (0,0) -- (-.35,-.35);
\draw[ red, line width=0.8mm] (0,0) -- (.35,-.35);
\draw[ blue, line width=0.8mm] (0,0) -- (0,-.5);
\draw[ blue, line width=0.8mm] (0,0) -- (-.35,.35);
\draw[ blue, line width=0.8mm] (0,0) -- (.35,.35);
\end{tikzcd}
&$\longmapsto$
& $0$;
\\
\\
\begin{tikzcd}
\draw[dashed,color=black!60] (0,0) circle (0.5);
\draw[ red, line width=0.8mm] (0,0) -- (0,0.5);
\draw[ red, line width=0.8mm] (0,0) -- (0.5,0);
\draw[ red, line width=0.8mm] (0,0) -- (0,-.5);
\draw[ red, line width=0.8mm] (0,0) -- (-0.5,0);

\draw[ blue, line width=0.8mm] (.35,.35) -- (-.35,-.35);
\draw[ blue, line width=0.8mm] (-.35,.35) -- (.35,-.35);
\end{tikzcd}
&$\longmapsto$
& $0$.
\end{tabular}
\caption{{\small The assignment of $G_0: \cD \ra \cB$ on generating morphisms. Note that $a_j,$ $b_j,$ $c_j,$ $d_j,$ $f^j_k$ $\in \R$ for $j = 1$ and $\in \C$ for $j > 1.$ }}
\label{fig: quotient functor on morphism}
\end{figure}
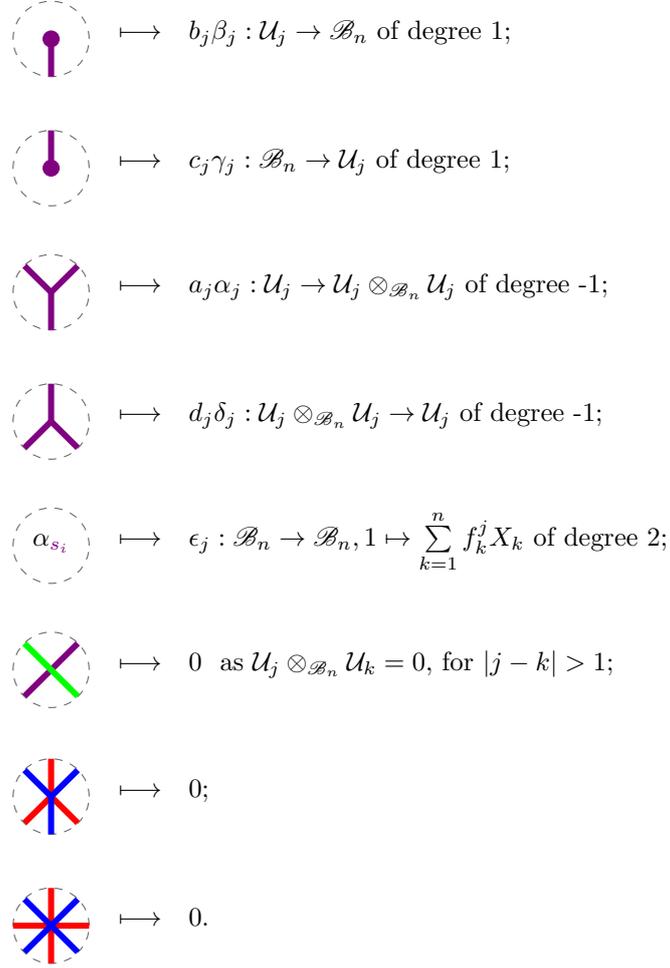

We want to find a set of scalars such that the restrictions imposed by the relations in $\cD$ are satisfied.
To illustrate, let us consider the barbell relation \eqref{barbell} for $s_1$:

\begin{equation*}
 \centering
\begin{tikzpicture}[scale=0.7]
\draw[dashed,color=black!60] (0,0) circle (1.0);
\draw[dashed,color=black!60] (3,0) circle (1.0);
\draw[ violet, line width=0.8mm] (0,.55) -- (0,-.55);
\filldraw[violet] (0,.55) circle (3pt) ;
\filldraw[violet] (0,-.55) circle (3pt) ;
\node at (1.5,0) {=};
\node at (3,0) {{\LARGE $\alpha_{\textcolor{violet}{s_1}}$}};
\end{tikzpicture}. 
\end{equation*}
By definition,
\begin{align*}
1 & \mapsto c_1 \left( X_1 \otimes e_1 + e_1 \otimes X_1 + (2|1) \otimes (1|2) + (-ie_2)(2|1) \otimes (1|2)(ie_2) \right)\\
& \mapsto c_1b_1 \left(  X_1 + X_1 + X_2 + X_2 \right) \\
& = c_1b_1 \left( 2X_1 + 2X_2 \right).
\end{align*}
\noindent Equating with the right hand side, we get $c_1b_1 \left( 2X_1 + 2X_2 \right) = \sum^n_{k=1} f^1_k X_k$ which, in turn, implies $2c_1 b_1 = f^1_1 $ and $2c_1 b_1 = f^1_2.$

  By the same token, let us look at the type B 8-valences relation \eqref{mst4}: suppose \textcolor{red}{red} (left-most strand) encodes \textcolor{red}{$s_1$} and \textcolor{blue}{blue} encodes \textcolor{blue}{$s_2$},
\begin{align*}
\begin{tikzcd}
\draw[color=red, line width=0.8mm] (-1.45,0) -- (.55,0);
\draw[color=red, line width=0.8mm] (0,0.75) -- (0,-0.75);
\filldraw[red]  (0.55,0) circle (3pt) ;
\draw[color=blue, line width=0.8mm] (.75,.75) -- (-.75,-0.75);
\draw[color=blue, line width=0.8mm] (-.75,.75) -- (.75,-0.75);
\end{tikzcd} 
 &= 
\begin{tikzcd}
\draw[red, line width=0.8mm] (.125,.68) -- (.125,-.68);
\draw[blue, line width=0.8mm] (-.125,.68) -- (-.125,-.68);
\draw[red, line width=0.8mm] (-.5,0) -- (-1,0);
\draw[blue, line width=0.8mm] (.45,0) -- (.7,0);
\draw[blue, line width=0.8mm] (.7,.68) -- (.7,-.68);
\filldraw[red] (-.425,0) circle (3pt) ;
\filldraw[blue] (.425,0) circle (3pt) ;
\end{tikzcd} 
- \frac{a_{s_1,s_2}}{a_{s_2,s_1}a_{s_1,s_2}-1} 
\begin{tikzcd}
\draw[red, line width=0.8mm] (-.5,0) -- (-1,0);
\draw[blue, line width=0.8mm] (.7,.68) -- (.7,-.68);
\draw[red, line width=0.8mm] (.23,0.4) -- (.23,0.66);
\draw[red, line width=0.8mm] (.23,-0.4) -- (.23,-0.66);
\draw[blue, line width=0.8mm] (0,0) -- (-.45,.58);
\draw[blue, line width=0.8mm] (0,0) -- (-.45,-.58);
\draw[blue, line width=0.8mm] (0,0) -- (.7,0);
\filldraw[red] (-.4,0) circle (3pt) ;
\filldraw[red] (.23,0.3) circle (3pt) ;
\filldraw[red] (.23,-0.3) circle (3pt) ;
\end{tikzcd} 
- \frac{a_{s_2,s_1}}{a_{s_2,s_1}a_{s_1,s_2}-1} 
\begin{tikzcd}
\draw[blue, line width=0.8mm] (.5,0) -- (.7,0);
\draw[blue, line width=0.8mm] (.7,.6) -- (.7,-.6);
\draw[blue, line width=0.8mm] (-.23,0.4) -- (-.23,0.66);
\draw[blue, line width=0.8mm] (-.23,-0.4) -- (-.23,-0.66);
\draw[red, line width=0.8mm] (0,0) -- (.45,.58);
\draw[red, line width=0.8mm] (0,0) -- (.45,-.58);
\draw[red, line width=0.8mm] (0,0) -- (-1,0);
\filldraw[blue] (.4,0) circle (3pt) ;
\filldraw[blue] (-.23,0.3) circle (3pt) ;
\filldraw[blue] (-.23,-0.3) circle (3pt) ;
\end{tikzcd}  \\
&+ \frac{1}{a_{s_2,s_1}a_{s_1,s_2}-1} 
\begin{tikzcd}
\draw[blue, line width=0.8mm] (.13,0) -- (.7,0);
\draw[blue, line width=0.8mm] (.7,.6) -- (.7,-.6);
\draw[blue, line width=0.8mm] (.13,0) -- (-.45,.58);
\draw[red, line width=0.8mm] (-.15,0) -- (.45,-.58);
\draw[red, line width=0.8mm] (-.15,0) -- (-1,0);
\draw[red, line width=0.8mm] (.23,0.4) -- (.25,0.66);
\draw[blue, line width=0.8mm] (-.23,-0.4) -- (-.25,-0.66);
\filldraw[red] (.23,0.3) circle (3pt) ;
\filldraw[blue] (-.23,-0.3) circle (3pt) ;
\end{tikzcd} 
+ \frac{1}{a_{s_2,s_1}a_{s_1,s_2}-1}
\begin{tikzcd}
\draw[blue, line width=0.8mm] (.13,0) -- (.7,0);
\draw[blue, line width=0.8mm] (.7,.68) -- (.7,-.68);
\draw[red, line width=0.8mm] (.23,-0.4) -- (.23,-0.66);
\draw[blue, line width=0.8mm] (-.23,0.4) -- (-.23,0.66);
\draw[red, line width=0.8mm] (-.13,0) -- (.45,.58);
\draw[red, line width=0.8mm] (-.13,0) -- (-1,0);
\draw[blue, line width=0.8mm] (.13,0) -- (-.45,-.58);
\filldraw[red] (.23,-0.3) circle (3pt) ;
\filldraw[blue] (-.23,0.3) circle (3pt) ;
\end{tikzcd} \\	
\end{align*}  
which, by definition of the functor, is equivalent to checking the equation
\begin{align*}
0 
\  &= \ 
\begin{tikzcd}
\draw[red, line width=0.8mm] (.125,.68) -- (.125,-.68);
\draw[blue, line width=0.8mm] (-.125,.68) -- (-.125,-.68);
\draw[red, line width=0.8mm] (-.425,0) -- (-.425,-0.68);
\draw[blue, line width=0.8mm] (.45,.68) -- (.45,-.68);
\filldraw[red] (-.425,0) circle (3pt) ;
\end{tikzcd} 
 \ - \frac{a_{s_1,s_2}}{a_{s_2,s_1}a_{s_1,s_2}-1}  \  \
\begin{tikzcd}
\draw[red, line width=0.8mm] (-.4,0) -- (-.4,-.68);
\draw[blue, line width=0.8mm] (.7,.68) -- (.7,-.68);
\draw[red, line width=0.8mm] (.3,0.4) -- (.3,0.68);
\draw[red, line width=0.8mm] (.3,-0.4) -- (.3,-0.68);
\draw[blue, line width=0.8mm] (-0.05,0) -- (-.05,.68);
\draw[blue, line width=0.8mm] (-0.05,0) -- (-.05,-.68);
\draw[blue, line width=0.8mm] (0,0) -- (.7,0);
\filldraw[red] (-.4,0) circle (3pt) ;
\filldraw[red] (.3,0.3) circle (3pt) ;
\filldraw[red] (.3,-0.3) circle (3pt) ;
\end{tikzcd} 
 \ - \frac{a_{s_2,s_1}}{a_{s_2,s_1}a_{s_1,s_2}-1} \
\begin{tikzcd}
\draw[blue, line width=0.8mm] (.5,.68) -- (.5,-.68);
\draw[blue, line width=0.8mm] (-.23,0.4) -- (-.23,0.68);
\draw[blue, line width=0.8mm] (-.23,-0.4) -- (-.23,-0.68);
\draw[red, line width=0.8mm] (0.15,0) -- (0.15,.68);
\draw[red, line width=0.8mm] (0.15,0) -- (0.15,-.68);
\draw[red, line width=0.8mm] (0.15,0) .. controls (-.5,-.05) and (-.58,-.1) .. (-.6,-0.68);
\filldraw[blue] (-.23,0.3) circle (3pt) ;
\filldraw[blue] (-.23,-0.3) circle (3pt) ;
\end{tikzcd}  \\
 & \ + \frac{1}{a_{s_2,s_1}a_{s_1,s_2}-1} \ 
\begin{tikzcd}
\draw[blue, line width=0.8mm] (.7,0) .. controls (.05, .05) and (-.07,.1) .. (-.05,.68);
\draw[blue, line width=0.8mm] (.7,.68) -- (.7,-.68);
\draw[red, line width=0.8mm] (-.05,0) .. controls (0,-.009) and (.3,-.08) .. (.35,-.68);
\draw[red, line width=0.8mm] (-.05,0) .. controls (-0.1,-.009) and (-.35,-.08) .. (-.4,-.68);
\draw[red, line width=0.8mm] (.35,0.4) -- (.35,0.66);
\draw[blue, line width=0.8mm] (-.05,-0.4) -- (-.05,-0.66);
\filldraw[red] (.35,0.3) circle (3pt) ;
\filldraw[blue] (-.05,-0.3) circle (3pt) ;
\end{tikzcd} 
\ + \frac{1}{a_{s_2,s_1}a_{s_1,s_2}-1} \ 
\begin{tikzcd}
\draw[blue, line width=0.8mm] (.7,0) .. controls (.05, -.05) and (-.07,-.1) .. (-.05,-.68);
\draw[red, line width=0.8mm] (-.05,0) .. controls (0,.009) and (.3,.08) .. (.35,.68);
\draw[red, line width=0.8mm] (-.05,0) .. controls (-0.1,-.009) and (-.35,-.08) .. (-.4,-.68);
\draw[blue, line width=0.8mm] (.7,.68) -- (.7,-.68);
\draw[red, line width=0.8mm] (.35,-0.4) -- (.35,-0.66);
\draw[blue, line width=0.8mm] (-.05,0.4) -- (-.05,0.66);
\filldraw[red] (.35,-0.3) circle (3pt) ;
\filldraw[blue] (-.05,0.3) circle (3pt);
\filldraw[black] (1,-.64) circle (.5pt);
\end{tikzcd}\\	
\end{align*}  
\noindent Observe that there are five terms of Soergel graphs on the right-hand side. One thing to note here is that $\cU_1 \otimes \cU_2 \otimes \cU_1 \otimes \cU_2$ is spanned by  $e_1 \otimes_\R (1|2) \otimes_{\Ba_n} e_2 \otimes_\C e_2 \otimes_{\Ba_n} (2|1) \otimes_\R (1|2) \otimes_{\Ba_n} e_2 \otimes_\C e_2  $ and $e_1 \otimes_\R (1|2) \otimes_{\Ba_n} e_2 \otimes_\C e_2 \otimes_{\Ba_n} (-ie_2)(2|1) \otimes_\R (1|2) \otimes_{\Ba_n} e_2 \otimes_\C e_2.$ 
 Without the coefficients,  looking at the second term in the above equation and applying appropriate,  we get
\begin{align*}
& e_1 \otimes (1|2) \otimes e_2 \otimes e_2 \otimes (2|1) \otimes (1|2) \otimes e_2 \otimes e_2  \\
& \ \ \xmapsto{\beta_1}  b_1  (1|2) \otimes e_2 \otimes (2|1) \otimes (1|2) \otimes e_2 \otimes e_2  \\
& \ \ \xmapsto{\beta_1} b_1^2  (1|2) \otimes X_2 \otimes e_2 \otimes e_2 \\
& \ \ \xmapsto{\delta_2} d_2 b_1^2  (1|2) \otimes e_2 \\
& \ \ \xmapsto{\alpha_2} a_2 d_2 b_1^2  (1|2) \otimes e_2 \otimes e_2 \otimes e_2 \\
& \ \ \xmapsto{\gamma_1} c_1 a_2 d_2 b_1^2  (1|2) \otimes e_2 \otimes (-ie_2)(2|1) \otimes (1|2)(ie_2) \otimes e_2 \otimes e_2  \\
& \ \  \ \ \ \ \ + c_1 a_2 d_2 b_1^2  (1|2) \otimes e_2 \otimes (2|1) \otimes (1|2) \otimes e_2 \otimes e_2.
\end{align*}
Similarly, 
\begin{align*}
& e_1 \otimes (1|2) \otimes e_2 \otimes e_2 \otimes (-ie_2)(2|1) \otimes (1|2) \otimes e_2 \otimes e_2  \\
& \ \ \mapsto c_1 a_2 d_2 b_1^2 \Big( (1|2) \otimes e_2 \otimes (2|1) \otimes (1|2)(-ie_2) \otimes e_2 \otimes e_2  \\
& \ \ \ \ \ + (1|2) \otimes e_2 \otimes (-ie_2)(2|1) \otimes (1|2) \otimes e_2 \otimes e_2 \Big).
\end{align*}
Once the calculations for all the five Soergel graphs has been done, comparing coefficients coming from four basis elements in the codomain will yield four defining equations.

Before giving you all the relations, we will deal with the unnecessary or overlapped relations. For \eqref{needle}, it says $b_j d_j a_j \beta_j \delta_j \alpha_j = 0$ which is true as $\delta_j \alpha_j = 0.$
 On the other hand, \eqref{Frobenius} (equiv. \eqref{assocmult} and \eqref{coassoccomult}) do not impose any restrictions on the coefficients whereas \eqref{wall} is replaced by \eqref{enddot counit comult} and \eqref{startdot unit multi}.
    In addition, the relations \eqref{assoc3}, \eqref{assoc4}, \eqref{Zamo4} \eqref{Zamo3}, \eqref{Zamo2}, \eqref{A3}, and \eqref{B3} are all trivially satisfied as the $2m_{st}$-valents vertices are killed for every $s,t \in S.$
    Finally, \eqref{mst2} has both side equal to zero as $\cU_j \otimes_{\Ba_n} \cU_k = 0$ in $\cB$ for $|j-k| > 1 $.

We now summarise all of the other required relations below (these are mostly the same as the type A case, with the exception of the two colour relations):

\eqref{barbell} \  $\implies$ \ \ $f^1_1 = 2b_1 c_1, f^1_2 =2 b_1 c_1, f^j_j = 2 b_j c_j, f^j_{j \pm 1} = b_j c_j, f^j_k=0$ for $j,k \geq 2$ and $|j-k| > 1,$

\eqref{polyforce} \ $\implies$  \  $f^1_1 = 2b_1 c_1, f^1_2 = - 2 b_2 c_2, f^j_j = 2 b_j c_j, f^{j \pm 1}_j = -b_j c_j, f^j_k=0$ for $j,k \geq 2$ and $|j-k| > 1,$

\eqref{mst3} (Type A 6-valences relation) \ $\implies$ \ $d_{j \pm 1} b_j c_j = -b_{j \pm 1}$ for suitable $j,$  

\eqref{mst4} (Type B 8-valences relation) \ $\implies$  \
Since $a_{s_2,s_1} a_{s_1,s_2}-1 = (-2)(-1) -1 =1,$ we will simplify the denominator first. 
For {\color{red}$s_1$ (red)} left-aligned, we get
$$b_1 - {a_{s_1,s_2}}b_1^2 d_2 a_2 c_1 - {a_{s_2,s_1}}b_2 d_1 c_2 + b_2 d_1 b_1 a_2 c_1 + b_1 d_2 c_2 = 0,$$
$$b_1 - a_{s_1,s_2} b_1^2 d_2 a_2 c_1 = 0, \ \
-a_{s_1,s_2} b_1^2 d_2 a_2 c_1 + b_2 d_1 b_1 a_2 c_1 = 0,  \ \
-a_{s_1,s_2} b_1^2 d_2 a_2 c_1 + b_1 d_2 c_2 = 0,$$
while for {\color{blue}$s_2$ (blue)} left-aligned, we get
$$b_2 - a_{s_2,s_1} b_2^2 d_1 a_1 c_2 - a_{s_1,s_2} b_1 d_2 c_1 + b_1 d_2 a_1 b_2 c_2 + b_2 d_1 c_1 = 0,$$
$$b_2 -  a_{s_1,s_2} b_1 d_2 c_1 = 0, \ \
 - a_{s_1,s_2} b_1 d_2 c_1 + b_1 d_2 a_1 b_2 c_2 = 0, \ \
-a_{s_1,s_2} b_1 d_2 c_1 + b_2 d_1 c_1 = 0, $$

\eqref{enddot counit comult} \ \  $\implies$ \ \ $a_jb_j = 1$  for all $1 \leq j \leq n,$

\eqref{startdot unit multi} \ \ $\implies$ \ \ $c_jd_j = 1$  for all $1 \leq j \leq n,$

\eqref{selfadjoint} \ \ $\implies$ \ \ $a_jb_jc_jd_j = 1$  for all $1 \leq j \leq n,$

\eqref{mult rot comult} \ \ $\implies$ \ \ $a_jb_jd_j = d_j$  for all $1 \leq j \leq n,$

\eqref{comult rot mult} \ \ $\implies$ \ \ $a_jc_jd_j = a_j$  for all $1 \leq j \leq n,$

\eqref{counit rot unit} \ \ $\implies$ \ \ $b_jc_jd_j = b_j$  for all $1 \leq j \leq n,$

\eqref{unit rot counit} \ \ $\implies$ \ \ $a_jb_jc_j = c_j$  for all $1 \leq j \leq n.$

The solution $a_j = b_j = (-1)^{j+1}, c_j = d_j = 1, f^1_2 = 2, f^j_j = (-1)^{j+1} 2, f^j_{j \pm 1} = (-1)^{j+1}$ and $f^j_k = 0$ for $|j-k|>1$ gives our desired functor.
\end{proof}

\begin{corollary} \label{factors}
The functor $G_0: \cD \to \cB$ in \cref{maintheorem} induces an exact monoidal functor $\ol{G}: \Kom^b(Kar(\ol{\cD})) \ra \Kom^b(Kar(\ol{\Ba}))$.
This functor $\ol{G}$ sends the generators of the 2-braid groups $F_{s_j}$ to $R_j[-1](1)$, matching Rouquier's complexes with our twist complexes (up to internal grading shift and cohomological shift).
\end{corollary}

\begin{corollary}[Faithfulness of type $B$ 2-braid group] 
The group homomorphism $\cA(B_n) \ra $ Pic$(2$-$\cB r)$ sending $\sigma_j \mapsto F_{s_j}$ is faithful.
\end{corollary}
\begin{proof}
By \cref{Cat B action}, the assignment $\sigma_j \mapsto R_j[-1](1)$ is faithful as the (gradings-shifted) action it induces on $\Kom^b(\Ba_n$-$p_r g_r mod)$ is faithful.
\cref{factors} shows that this assignment factors through Pic$(2$-$\cB r)$, which implies that the group homomorphism $\cA(B_n) \ra $ Pic$(2$-$\cB r)$ is also faithful.
\end{proof}

\begin{remark}
Note that a similar proof strategy on the ``decategorified level'' will not work, as the Burau representation may not be faithful even in type $A_n$; for $n=5$ see \cite{Bige_Burau}.
To the best of our knowledge, the faithfulness of the Hecke algebra representation of Artin braid groups remains open for large ranks.
\end{remark}

\printbibliography

\end{document}